\documentclass[11pt]{amsart}

\usepackage{amsmath, amsthm, amssymb, amsfonts, dsfont, enumerate}
\usepackage[colorlinks=true,linkcolor=blue,urlcolor=blue]{hyperref}
\usepackage{color}
\usepackage{geometry}
\usepackage{mathtools}

\mathtoolsset{showonlyrefs}

\newtheorem{theorem}{Theorem}[section]
\newtheorem{remark}[theorem]{Remark}
\newtheorem{assumption}[theorem]{Assumption}
\newtheorem{lemma}[theorem]{Lemma}
\newtheorem{proposition}[theorem]{Proposition}
\newtheorem{corollary}[theorem]{Corollary}

\def \cA{{\mathcal A}}

\def \cC{{\mathcal C}}

\def \cF{{\mathcal F}}
\def \cG{{\mathcal G}}
\def \cI{{\mathcal I}}

\def \cL{{\mathcal L}}
\def \cO{{\mathcal O}}

\def \cS{{\mathcal S}}

\def \E{\mathsf{E}}
\def \EQ{\mathsf{E}^{\mathsf{Q}}}
\def \P{\mathsf{P}}
\def \Q{\mathsf{Q}}
\def \R{\mathbb{R}}

\def \eps{\varepsilon}

\def \v{{\mathbf v}}
\def \XX{{\widehat X}}
\def \YY{{\widehat Y}}
\def \UU{{\widehat U}}
\def \uu{{\widehat u}}
\def \pp{{\widehat \pi}}
\def \PH{{\widehat \Phi}}
\def \OP{{\overline \P}}
\def \OE{{\overline \E}}
\def \OW{{\overline W}}
\def \OU{{\overline U}}
\def \DD{{\widehat D}}

\title[Dividend with partial information]{Optimal dividends with partial information \\ and stopping of a degenerate reflecting diffusion}
\author[T.~De Angelis]{Tiziano De Angelis}

\address{T.~De Angelis: School of Mathematics, University of Leeds, Woodhouse Lane, LS2 9JT Leeds, UK.}
\email{\href{mailto:t.deangelis@leeds.ac.uk}{t.deangelis@leeds.ac.uk}}
\date{\today}

\numberwithin{equation}{section}

\begin{document}

\begin{abstract}
We study the optimal dividend problem for a firm's manager who has partial information on the profitability of the firm. The problem is formulated as one of singular stochastic control with partial information on the drift of the underlying process and with absorption. In the Markovian formulation, we have a 2-dimensional degenerate diffusion, whose first component is singularly controlled and it is absorbed as it hits zero. The free boundary problem (FBP) associated to the value function of the control problem is challenging from the analytical point of view due to the interplay of degeneracy and absorption. We find a probabilistic way to show that the value function of the dividend problem is a smooth solution of the FBP and to construct an optimal dividend strategy. Our approach establishes a new link between multidimensional singular stochastic control problems with absorption and problems of optimal stopping with `creation'. One key feature of the stopping problem is that creation occurs at a state-dependent rate of the `local-time' of an auxiliary 2-dimensional reflecting diffusion.  
\end{abstract}

\maketitle

\section{Introduction}

We study a singular stochastic control problem on a linearly controlled, 1-dimensional Brownian motion $X$ with (random) drift $\mu$. The problem is motivated by the dividend problem, where $X$ denotes the revenues of a firm and the firm's manager needs to distribute dividends to the share-holders in an optimal way but being mindful of the risk of default. Similarly to the existing literature, we account for the risk of default by letting the process $X$ be absorbed upon reaching zero. 

As one may expect, the optimal distribution of dividends is very sensitive to the profitability of the firm, which is encoded in the drift $\mu$ of the process $X$. A positive drift reflects a company in good health and, as a rule of thumb, dividends are paid when revenues are sufficiently high (which is expected to occur rather often) so to keep a low risk of default. On the contrary, a negative drift indicates a firm that operates at a loss and therefore should be wound up as soon as possible by paying out all dividends. 

Estimating profitability is a challenging task in many real-world situations which has already received attention in the mathematical economic literature; see, e.g., \cite{DMV} for investment timing, \cite{DS} for contract theory, \cite{DG} for asset trading. In order to capture this feature in a non-trivial but tractable way, we assume partial information on the drift of the process $X$. This is a novelty compared to existing models on dividend distribution.

We remark that statistical estimation of the drift of a drifting Brownian motion from observation of the process is a much less efficient procedure than estimation of its volatility. Indeed, over a given period of time $[0,T]$, the variance on the classical estimator for the volatility can be reduced by increasing the number of observations, whereas this is not the case for the variance on the estimator for the drift $\mu$. The latter depends on $1/T$ (see \cite[Example 2.1]{EL11} for a simple example), hence an accurate estimate of $\mu$ requires a long period of observation under the exact same market conditions, which in reality is not feasible.   

Our study shows how the flow of information affects the firm's manager optimal dividend strategy: as in the informal discussion above, dividends are paid only when revenues exceed a critical value $d^*$, however, in contrast to the existing literature this critical value changes dynamically according to the manager's current belief on the profitability of the firm; as we will explain in more detail below, such belief is described by a state variable $\pi\in(0,1)$, where a value of $\pi$ close to 1 indicates a strong belief in a positive drift and a value of $\pi$ close to 0 indicates a strong belief in a negative drift; we observe that the critical value of the revenues $d^*$ increases (but stays bounded) as $\pi$ increases, which is in line with the intuition that a firm with high profitability expects good performance and chooses to pay dividends when large revenues are realised, so that the risk of default is kept low and the business can be sustained over longer times; on the contrary, if there is a weak belief in the profitability of the firm, then dividends will be paid also for lower levels of the revenues, as there is no expectation that these will increase in the future. The partially informed manager of our firm, learns about the true value of profitability by observing the stream of revenues $X$ and adjusts her strategy accordingly, so that dividends are paid dynamically at different levels of revenue depending on the learning process.

The observation of $X$ will in the end reveal the true drift $\mu$ so that the belief of the firm's manager will eventually converge to either $\pi=0$ or $\pi=1$. Her dividend strategy will then converge to the corresponding strategy for the problem with full information (see Proposition \ref{prop:lims-d}). This shows that our model complements and extends the existing literature, which will be reviewed in the next section, by displaying a richer structure of the optimal solution and by effectively adding a new dimension to the classical problem (i.e., the belief). For a broader discussion on the economic foundations and implications of a dividend problem with partial information we also refer the reader to the introduction of the preprint \cite{DV}, where a special case of our problem is studied with different methods (a detailed comparison is given in the final three paragraphs of the next section).

\subsection{Mathematical background and overview of main results}
Our specific mathematical interest is in the explicit characterisation of the optimal control in terms of an optimal boundary arising from an associated free boundary problem. To the best of our knowledge the study of free boundaries for  singular stochastic control problems associated to diffusions with \emph{absorption} and \emph{partial information} has never been addressed in the literature. Recently \O ksendal and Sulem \cite{OS12} studied general maximum principles for singular control problems with partial information. Their approach relies mostly on backward stochastic differential equations (BSDEs) and they provide general abstract results linking the value of the singular control problem to the solution of suitable BSDEs. Here instead we focus on a specific problem with the aim of a more detailed study of the optimal control. It is worth noticing that \cite{OS12} does not consider the case of absorbed diffusions, which is a source of interesting mathematical facts in our paper, as we will discuss below.

For the sake of tractability we choose a model in which $\mu$ is a random variable that can only take two real values, i.e.~$\mu\in\{\mu_0,\mu_1\}$, with $\mu_0<\mu_1$. The company's revenue, net of dividend payments, at time $t$ reads
\begin{align}\label{dyn}
X^D_t=X_t-D_t:=x+\mu t+\sigma B_t-D_t
\end{align}
where $B$ is a Brownian motion, $\sigma>0$, and $D_t$ denotes the total amount of dividends paid up to time $t$ (notice that $D$ is a non-decreasing process and we choose it to be right-continuous).
As in the most canonical formulation of the dividend problem, the insurance company's manager wants to maximise the discounted flow of dividends until the firm goes bankrupt. Moreover, the manager can infer the true value of $\mu$ by observing the evolution of $X$. 

Using filtering techniques the problem can be written in a Markovian framework by considering simultaneously the dynamics of $X^D$ and of the process $\pi_t:=\P(\mu=\mu_1|\cF^X_t)$, where $\cF^X_t=\sigma(X_s,\,s\le t)$. This approach has a long and venerable history in optimal stopping theory, with early contributions dating back to work of Shiryaev in 1960s in the context of quickest detection (see \cite{S10} for a survey. See also \cite{PJ17} for some recent developments and further references). However, it seems that such model has never been adopted in the context of singular control. 

One difficulty that arises by the reduction to Markovian framework is that the dynamic of the state process is two dimensional and diffusive. This leads to a variational formulation of the stochastic control problem in terms of PDEs and therefore explicit solutions cannot be provided, in general. 

The literature on the optimal dividend problem is very rich with seminal mathematical contributions by Jeanblanc and Shiryaev \cite{JS} and Radner and Shepp \cite{RS}. More recent contributions include, among many others (see, e.g., the survey \cite{A09}), \cite{AGRS} and \cite{E15} who consider random interest rates, \cite{APP07} who allow for jumps in the dynamic of $X$, \cite{JP12} who consider a regime switching dynamic for the coefficients in \eqref{dyn}, \cite{BKY13} who consider jumps in the dynamic of $X$ and fixed transaction costs for dividend lump payments. However, research so far has largely focused on explicitly solvable examples. This means that, in the largest majority of papers, the underlying stochastic dynamics are either one dimensional, or two dimensional but with one of the state processes driven by a Markov chain. Moreover, the time horizon $\gamma^D$ of the optimisation is usually assumed to be the first time of $X^D$ falling below some level $a\ge 0$. Alternatively, capital injection is allowed and the opitimisation continues indefinitely, i.e. $\gamma^D=+\infty$. These choices of $\gamma^D$ make the problem time homogenous and easier to deal with. In absence of capital injection, even just assuming a finite time-horizon for the dividend problem, i.e.~taking $\gamma^D\wedge T$ for some deterministic $T>0$, introduces major technical difficulties. The latter were addressed first in \cite{G1} and \cite{G2} with PDE methods, and then in \cite{DeAE17} with probabilistic methods. Interestingly, the finite time-horizon is more easily tractable in presence of capital injection, as shown in \cite{FS18} using ideas originally contained in \cite{ElKK88}.

Here we take the approach suggested in \cite{DeAE17} but, as we will explain below, we substantially expand results therein. First we link our dividend problem to a suitable optimal stopping one. Then we solve the optimal stopping problem (OSP) by characterising its optimal stopping rule in terms of a free boundary $\pi\mapsto d(\pi)$. Finally, we deduce from properties of the value function $U$ of the OSP that the value function $V$ of the dividend problem is a strong solution of an associated variational inequality on $\R_+\times[0,1]$ with gradient constraint. Moreover, using the boundary $d(\cdot)$ we express the optimal dividend strategy as an explicit process depending on $t\mapsto d(\pi_t)$. It is worth noticing that we can prove that $V\in C^1(\R_+\times(0,1))$, with $V_{xx}$ and $V_{x\pi}$ belonging to $C(\R_+\times(0,1))$ and $V_{\pi\pi}\in L^\infty(\R_+\times (0,1))$. This type of global regularity cannot be easily obtained with PDE methods due to the degeneracy of the underlying diffusion.
Here we obtain these results with a careful probabilistic study of the value function $U$. In particular the argument used to prove $V_{\pi\pi}\in L^\infty(\R_+\times (0,1))$ in Proposition \ref{prop:vpipi} seems completely new in the related literature.

As in \cite{DeAE17} the presence of an absorbing point for the process $X^D$ `destroys' the standard link between optimal stopping and singular control. Such link has been studied by many authors: Bather and Chernoff \cite{BC67} and Bene{s}, Shepp and Witsenhausen \cite{BSW80} were the first to observe it and Taksar \cite{Taksar85} provided an early connection to Dynkin games. Extensions and refinements of the initial results were obtained in a long series of subsequent papers using different methodologies. Just to mention a few we recall \cite{BK}, \cite{ElKK88} and \cite{KS84} who address the problem with probabilistic methods, \cite{BenthReikvam} who use viscosity theory, \cite{GT08} who link singular control problems to switching problems. 

Departing from the literature mentioned above, here we prove that $V_x=U$ where now $U$ is the value function of an OSP whose underlying process is a 2-dimensional, uncontrolled, degenerate diffusion $(\XX,\pp)$, which lives in $\R_+\times[0,1]$ and is reflected at $\{0\}\times(0,1)$, towards the interior of the domain, along the direction of a state-dependent vector $\v(\pp)$ (see Section \ref{sec:stopping}). Moreover, upon each reflection, the gain process that is underlying the OSP increases exponentially at a rate that depends on the `intensity' of the reflection and on the value of the process $\pp_t$. We call this behaviour of the gain process: `state-dependent creation' of the process $(\XX,\pp)$ at $\{0\}\times(0,1)$ (cf.~\cite{Pe14}). Indeed it is interesting that the `creation' feature of our reflected process links our paper to work by Stroock and Williams \cite{SW05} and Peskir \cite{Pe14}, concerning a type of non-Feller boundary behaviour of 1-dimensional Brownian motion with drift. Notice however, that in those papers the creation rate is constant and the problem is set on the real line, so that the direction of reflection is fixed. Here instead we deal with an example of a non-trivial, two dimensional, extension of the problem studied in \cite{SW05} and \cite{Pe14}.

A striking difference with the problem studied in \cite{DeAE17} is the much more involved dynamics underlying the OSP and the behaviour of the gain process. In \cite{DeAE17} the state dynamics in the control problem is of the form $(t,\check{X}^D_t)$, with $\check{X}^D$ as in \eqref{dyn} but with deterministic constant drift. This leads to an optimal stopping problem involving a 1-dimensional Brownian motion with drift which is reflected at zero, and which is created (in the same sense as above) at a constant rate. The state variable `time' is unaffected by the link between the dividend problem and the stopping one. Here instead, the correlation in the dynamics of $X^D_t$ and $\pi_t$ in the control problem induces two main effects: (i) it causes for the reflection of the process 
$(\XX,\pp)$ to be along the stochastic vector process $t\mapsto\v(\pp_t)$ (see \eqref{XX}--\eqref{pp}), (ii) it generates a non-constant, creation rate that depends on the process $\pp$ (see \eqref{U}).

The reflection of $(\XX,\pp)$ at $\{0\}\times(0,1)$ is realised by an increasing process $(A_t)_{t\ge0}$ which we can write down explicitly (see \eqref{A}) and which we will informally refer to as `local-time' of $(\XX,\pp)$ at $\{0\}\times(0,1)$. Despite its use in solving the dividend problem, the OSP that we derive is interesting in its own right and belongs to a class of problems that, to the best of our knowledge, has never been studied before. In particular this is an optimal stopping problem on a multi-dimensional diffusions, reflected in a domain $\cO$, with a gain process that increases exponentially at a rate proportional to the local time spent by the process in some portions of $\partial\cO$ (moreover such rate is non-constant). 

In conclusion, we believe that the main mathematical contributions of our work are the following: (i) for the first time we characterise the free boundary associated to a singular stochastic control problem with partial information on the drift of the process and absorption, (ii) we obtain rather strong regularity results for the value $V$ of the control problem, despite degeneracy of the associated HJB operator, (iii) we find a non-trivial connection between singular control for multi-dimensional diffusions with absorption, and optimal stopping of reflected diffusions with `state-dependent creation', (iv) we solve an example of a new class of optimal stopping problems, whose popularity we hope will increase with the increasing understanding of their role in the dividend problem. 

After completing this work we learned about the preprint \cite{DV} where the same problem is addressed in the special case of $\mu_1=-\mu_0$. In that setting the problem's dimension can be reduced by a transformation that makes one of the two state processes purely controlled (a closer inspection reveals that this is in line with the case of a null drift in our \eqref{YY2}). The problem in \cite{DV} can be solved by `guess-and-verify' via a parameter-dependent family of ODEs with suitable boundary conditions. The methods of \cite{DV} cannot be used for generic $\mu_0$ and $\mu_1$ because the dimension reduction is impossible and the ODE becomes a 2-dimensional free boundary problem involving partial derivatives. 

Besides the methodolocial differences between the two papers, the optimal strategy obtained in \cite{DV} shares similarities with ours but it also features a remarkable difference. Due to the fact that one of the state variables is purely controlled, in \cite{DV} the level of future revenues at which dividends will be paid can only increase after each dividend payment. As stated in \cite{DV}, this can be understood as the firm's manager `becoming more confident about the relevance of their project'. When revenues reach a new maximum, this suggests to the manager that the drift be positive; however, the symmetric structure $\mu_1=-\mu_0$ is such that she does not subsequently change her view, even if revenues start fluctuating downwards. This fact stands in sharp contrast with our solution, which instead allows the manager to increase/decrease her revenues target level depending on the new information acquired.

The rest of the paper is organised as follows. In Section \ref{sec:setting} we cast the problem and provide its Markovian formulation. Section \ref{sec:verif} introduces the verification theorem which we aim at proving probabilistically in the subsequent sections. The main technical contribution of the paper is contained in Sections \ref{sec:stop}, \ref{sec:fine} and \ref{sec:solution}. In the first part of Section \ref{sec:stop} we introduce the stopping problem for a 2-dimensional degenerate diffusion with state-dependent reflection. Then, in the rest of Section \ref{sec:stop} and in Section \ref{sec:fine}, we study properties of the associated value function and obtain geometric properties of the optimal stopping set. In Section \ref{sec:solution} we prove that the value function and the optimal control of the dividend problem can be constructed from the value function of the optimal stopping problem and its optimal stopping region. A short appendix contains a rather standard proof of the verification theorem stated in Section \ref{sec:verif}.

\section{Setting}\label{sec:setting}
We consider a complete probability space $(\Omega,\cF,\P)$ equipped with a 1-dimensional Brownian motion $(B_t)_{t\ge0}$ and its natural filtration $(\cF^B_t)_{t\le0}$ completed with $\P$ null sets. On the same probability space we also have a random variable $\mu$ which is independent of $B$ and takes two possible real values $\mu_0<\mu_1$, with probability $\P(\mu=\mu_1)=\pi\in[0,1]$. Further, given $x>0$ and $\sigma>0$, we model the firm's revenue in absence of dividend payments by the process $(X_t)_{t\ge0}$ defined as
\begin{align}\label{X0}
X_t=x+\mu t+\sigma B_t\quad t\ge0.
\end{align}  
We denote by $(\cF^{X}_t)_{t\ge0}$ the filtration generated by $X$ and we say that a dividend strategy is a $(\cF^{X}_t)_{t\ge0}$-adapted, increasing, right-continuous process $(D_t)_{t\ge 0}$ with $D_{0-}=0$. In particular $D_t$ represents the cumulative amount of dividends paid by the firm up to time $t$ and we say that the firm's profit, under the dividend strategy $D$, is
\begin{align}\label{XD}
X^D_t=x+\mu t+\sigma B_t-D_t,\quad \text{for $t\ge0$}.
\end{align}  
Notice that for $D\equiv 0$ we formally have $X^0=X$. As it is customary in the dividend problem, we define a default time at which the firm stops paying dividends and we denote it by
\begin{align}
\gamma^D:=\inf\{t\ge0\,:\,X^D_t\le 0\}.
\end{align}

Equipped with this simple model for the firm's profitability, the manager of the firm wants to maximise the expected flow of discounted dividends until the default time, where discounting occurs at a constant rate $\rho>0$, i.e.:
\begin{align}\label{P1}
\text{Maximise the value of}\:\:\E\left[\int^{\gamma^D}_{0-}e^{-\rho t}dD_t\right]\:\:\text{over $D\in\cA$},
\end{align}
where $\cA$ denotes the set of admissible dividend strategies. In particular 
\begin{align}\label{cA}
\text{$D\in\cA$ iff $D$ is $(\cF^{X}_t)_{t\ge0}$-adapted, increasing, right-continuous,}\\
\text{with $D_{0-}=0$ and such that $D_t-D_{t-}\le X^D_{t-}$ for all $t\ge0$, $\P$-a.s.}\nonumber
\end{align}

It is important to notice that the drift of $X^D$ is not affected by the choice of $D$, so that $X=X^D+D$. Moreover, the control process $D$ is chosen by the firm's manager based on their observation of the process $X$ and it is therefore natural that $D_t$ should be $\cF^X_t$-measurable. 

It is well known that the dynamic \eqref{XD} may be rewritten in a more tractable Markovian form, thanks to standard filtering methods (see, e.g., \cite[Sec.~4.2]{Sh}). In particular, denoting $\pi_t:=\P(\mu=\mu_1\big|\cF^{X}_t)$, one can construct a $((\cF^X_t)_{t\ge 0},\P)$-Brownian motion $(W_t)_{t\ge0}$ and write the dynamics of the couple $(X^D_t,\pi_t)_{t\ge0}$ for all $t>0$ in the form 
\begin{align}
\label{Xpi1}& dX^D_t=[\mu_0+\hat \mu\, \pi_t] d t + \sigma dW_t-dD_t & X^D_{0-}=x,\\ 
\label{Xpi1b}& d\pi_t=\theta\pi_t(1-\pi_t)dW_t & \pi_0=\pi,
\end{align}
under the measure $\P$, with $\hat \mu:=\mu_1-\mu_0$ and $\theta:=\hat \mu/\sigma$. We notice that \eqref{Xpi1} can be obtained from \eqref{XD} by formally replacing $\mu$ with $\E[\mu|\cF^X_t]$. Moreover,  $(\pi_t)_{t\ge 0}$ in \eqref{Xpi1b} is a bounded martingale, hence it is a martingale on $[0,\infty]$ and, in particular, $\pi_\infty\in\{0,1\}$ since all information is revealed at time $t=\infty$. 

Intuitively, we can say that at any given time $t\ge 0$ the amount of new information which becomes available to the firm's manager is measured by the absolute value of the increment $\Delta \pi_t$. Then, the learning rate depends on the so-called {\em signal-to-noise} ratio $\theta$ and on the current belief $\pi_t$, which appear in the diffusion coefficient in \eqref{Xpi1b}. Given an increment $\Delta W_t$ of the Brownian motion, the value of $|\Delta \pi_t|$ is increasing in the signal-to-noise ratio, as expected. Further, the maximum of the diffusion coefficient (hence the maximum learning rate) occurs when $\pi_t=1/2$, which corresponds to the most uncertain situation.     

Since $(X^D_t,D_t,\pi_t,W_t)_{t\ge 0}$ is $(\cF^X_t)_{t\ge 0}$-adapted and we do not need to consider any other filtration, from now on we denote $\cF_t=\cF^X_t$ to simplify the notation.
In the new Markovian framework our problem \eqref{P1} reads 
\begin{align}\label{P2}
V(x,\pi)=\sup_{D\in\cA}\E_{x,\pi}\left[\int^{\gamma^D}_{0-}e^{-\rho t}dD_t\right]\quad \text{for all}\:\:(x,\pi)\in\R_+\times(0,1),
\end{align}
where $\E_{x,\pi}[\,\cdot\,]:=\E[\,\cdot\,|X_0=x,\pi_0=\pi]$.

The formulation in \eqref{P2} of the optimal dividend problem with partial information corresponds to a singular stochastic control problem involving a 2-dimensional degenerate diffusion which is killed upon leaving the set $\R_+\times(0,1)$ (recall that if $\pi_0\in(0,1)$ then $\pi_t\in(0,1)$ for all $t\in(0,+\infty)$, whereas if $\pi_0\in\{0,1\}$ then $\pi_t=\pi_0$ for all $t>0$). In the economic literature, the value function $V$ of \eqref{P2} is traditionally considered as the value of the firm itself.

\begin{remark}\label{rem:intuition}
The case of full information corresponds to $\pi\in\{0,1\}$. In this case it is known that if the drift $\mu\le0$ it is optimal to pay all dividends immediately and liquidate the firm. On the contrary, if $\mu>0$ then dividends should be paid gradually according to a strategy characterised by a Skorokhod reflection of the process $X^D$ against a positive (moving) boundary (see \cite{JS} for the stationary case and \cite{DeAE17} for the non-stationary one). 

In our setting with partial information, it is clear that $\mu_0<\mu_1\le 0$ would lead to an immediate liquidation of the firm. The cases $\mu_1>\mu_0\ge 0$ and $\mu_0<0<\mu_1$ instead need to be studied separately as they present subtle technical differences which would make a unified exposition rather lengthy. In this paper we start with the case $\mu_0<0<\mu_1$, which seems economically the most interesting as it represents the uncertainty of a firm who cannot predict exactly whether its line of business is following an increasing or decreasing future trend. 
\end{remark}

Motivated by the remark above we make the following standing assumption throughout the paper: 
\begin{assumption}\label{ass:mu1}
We have $\mu_1>0>\mu_0$.
\end{assumption}

We close this section by introducing the infinitesimal generator $\cL_{X,\pi}$ associated to the uncontrolled process $(X_t,\pi_t)_{t\ge0}$. For functions $f\in C^2(\R_+\times[0,1])$ we have
\begin{align}\label{eq:L}
(\cL_{X,\pi} f)(x,\pi) :=&\tfrac{1}{2} \left(\sigma^2\, f_{xx} + 2\sigma\theta\pi(1-\pi)f_{x\pi}+ \theta^2\pi^2(1-\pi)^2f_{\pi\pi}\right)(x,\pi) \nonumber\\
&+ (\mu_0+\hat\mu\,\pi)f_x(x,\pi),
\end{align}
for $(x,\pi)\in\R_+\times[0,1]$ and where $f_{xx},f_{x\pi},f_{\pi\pi}$ are second derivatives and $f_x$ a first derivative.
For simplicity in the rest of the paper we also define
\begin{align}\label{cO}
\cO:=(0,+\infty)\times(0,1).
\end{align}
Moreover, given a set $A$ we denote by $\overline A$ its closure.

Following the approach introduced in \cite{DeAE17}, in the next section we will start our analysis by providing a verification theorem for $V$. Then we will use the latter to conjecture an optimal stopping problem that should be associated with $V_x$. It will soon become clear that the construction of \cite{DeAE17} is substantially easier than the one needed here. Our new construction also leads to a much more involved optimal stopping problem. 

\section{A verification theorem}\label{sec:verif}

A familiar heuristic use of the dynamic programming principle suggests that for any admissible control $D$ the process
\begin{align}\label{Vmg}
t\mapsto e^{-\rho (t\wedge\gamma^D)}V(X^D_{t\wedge\gamma^D},\pi_{t\wedge\gamma^D})+\int_{0-}^{t\wedge\gamma^D}e^{-\rho s}dD_s
\end{align}
should be a super-martingale and, if $D=D^*$ is an optimal control, then \eqref{Vmg} should be a martingale. Moreover, given a starting point $(x,\pi)$ one strategy could be to pay immediately a small amount $\delta$ of dividends, hence shifting the dynamics to the point $(x-\delta,\pi)$, and then continue optimally. Since this would in general be sub-optimal, one has 
\[
V(x,\pi)\ge V(x-\delta,\pi)+\delta\implies V_x(x,\pi)\ge 1.
\] 
If the inequality is strict, then the suggested strategy is strictly sub-optimal. Hence, the firm should pay dividends when $V_x=1$ and do nothing when $V_x> 1$. It is also clear from \eqref{P2} that $V(0,\pi)=0$ for all $\pi\in[0,1]$.

Based on this heuristic we can formulate the following verification theorem. Its proof is rather standard (see, e.g., \cite[Thm.~4.1, Ch.~VIII]{FS}) and we give it in appendix for completeness.
\begin{theorem}\label{thm:verif}
Let $v\in C^1(\cO)\cap C(\overline{\cO})$ with $v_{xx},\,v_{x\pi}\in C(\cO)$ and $v_{\pi\pi}\in L^\infty_{loc}(\cO)$. Assume that $0\le v(x,\pi)\le c\, x$, for all $(x,\pi)\in\cO$ and some $c>0$, and that it solves
\begin{align}
\label{HJB1}&\max\left\{(\cL_{X,\pi}-\rho)v,1-v_x\right\}(x,\pi)=0,&\text{for a.e.~$(x,\pi)\in\cO$}\\
\label{HJB2}&v(0,\pi)=0,&\text{for $\pi\in[0,1]$}.
\end{align}
Then $v\ge V$ on $\cO$. 

Let us denote 
\begin{align}\label{Iv}
\cI_v:=\{(x,\pi)\in\cO\,:\,v_x(x,\pi)>1\}.
\end{align}
In addition to the above assume that: $v\in C^2(\overline{\cI_v}\cap\cO)$ and there exists $D^{v}\in\cA$ such that, $\P_{x,\pi}$-almost surely for all $0\le t\le \gamma^{D^{v}}$, we have
\begin{align}
\label{SK-0}& (X^{D^{v}}_t,\pi_t)\in\overline{\cI_v},\\
\label{SK-1}& dD^{v}_t=\mathds{1}_{\{(X^{D^{v}}_{t-},\pi_t)\notin \,\cI_v\}}dD^{v}_t,\\
\label{SK-2}&\int_0^{\Delta D^v_t}\mathds{1}_{\{(X^{D^{v}}_{t-}-z,\pi_t)\in \,\cI_v\}}dz=0.
\end{align}
Then $V=v$ on $\cO$ and $D^*:=D^{v}$ is an optimal dividend strategy.
\end{theorem} 

From now on we will denote the inaction set for problem \eqref{P2} by $\cI$, and if $V\in C^1(\cO)$ this will correspond to the set
\begin{align}
\label{inaction}
\cI:=\{(x,\pi)\in\cO\,:\,V_x(x,\pi)>1\}.
\end{align}
For future reference we also recall that if $V\in C^2(\cO)$ solves \eqref{HJB1}, then in particular we have
\begin{align}
\label{harmonic}
(\cL_{X,\pi} V-\rho V)(x,\pi)=0,\qquad (x,\pi)\in\cI.
\end{align}
 

\section{Stopping a 2-dimensional diffusion with reflection and creation}\label{sec:stop}

In this section we will construct an optimal stopping problem (OSP) which involves a 2-dimensional degenerate diffusion. Such diffusion is kept inside $\cO$ by reflection at $\{0\}\times(0,1)$ and it is also \emph{created} upon each new reflection, in a sense which will be mathematically clarified later. Here we will also start a detailed study of the optimal stopping region and of the value function of such OSP, which will be instrumental to solve problem \eqref{P2}. 


\subsection{Construction of the stopping problem}\label{sec:stopping}
 
Let us assume for a moment that $V\in C^2(\overline{\cO})$ so that the boundary condition $V(0,\pi)=0$ would also imply $V_{\pi}(0,\pi)=V_{\pi\pi}(0,\pi)=0$. Then, for all $\pi\in(0,1)$ for which $(0,\pi)\in\cI$ (see \eqref{inaction}) we get from \eqref{harmonic}
\begin{align}
\label{elast-1}
\tfrac{\sigma^2}{2}V_{xx}(0,\pi)+\sigma\theta\pi(1-\pi)V_{x\pi}(0,\pi)+(\mu_0+\hat \mu \pi)V_x(0,\pi)=0.
\end{align}
Setting $u:=V_x$ we notice that $\cI=\{(x,\pi)\in\cO\,:\,u(x,\pi)>1\}$ and that $u\ge 1$ in $\cO$. Moreover, formally differentiating \eqref{harmonic} and using \eqref{elast-1} we obtain that $u$ solves
\begin{align}
\label{HJBu1}&(\cL_{X,\pi} u-\rho u)(x,\pi)=0 & (x,\pi)\in\cI\\
\label{HJBu2}&u(x,\pi)=1 & (x,\pi)\in\partial\cI\\
\label{HJBu3}&\tfrac{\sigma^2}{2}u_{x}(0,\pi)+\sigma\theta\pi(1-\pi)u_{\pi}(0,\pi) &\\
 &\qquad\qquad\:\:+(\mu_0+\hat \mu \pi)u(0,\pi)=0 & \text{for $\pi\in(0,1)$ s.t.~$(0,\pi)\in\cI$}.\nonumber
\end{align}
We claim that the variational problem \eqref{HJBu1}-\eqref{HJBu2}-\eqref{HJBu3} should be connected to the optimal stopping problem \eqref{U} given below. First we state the problem, then we give a heuristic justification of our claim and finally we prove, in several steps, that our conjecture is indeed correct.

Let $(\XX,\pp)$ be solution of the system, for $t>0$,
\begin{align}
\label{XX}&d\XX_t=(\mu_0+\hat\mu\pp_t)dt+\sigma dW_t+dA_t, & \XX_0=x\\[+3pt]
\label{pp}&d\pp_t=\theta\pp_t(1-\pp_t)(dW_t+\tfrac{2}{\sigma}dA_t),& \pp_0=\pi
\end{align}
where $(A_t)_{t\ge0}$ is an increasing continuous process, started at time zero from $A_0=0$ and such that $\P$-a.s
\begin{align}\label{SK-A}
\XX_t\ge 0\quad\text{and}\quad dA_t=\mathds{1}_{\{\XX_t=0\}}dA_t\quad\text{for all $t\ge0$.} 
\end{align}
Notably the process $(\XX,\pp)$ is a 2-dimensional degenerate diffusion which is reflected at $\{0\}\times(0,1)$ towards the interior of $\cO$, along the state-dependent unitary vector 
\[
\v(\pi) := \left ( \frac{1}{c(\pi)},\frac{\tfrac{2\theta}{\sigma} \pi(1-\pi)}{c(\pi)}\right)\quad\text{with $c(\pi):=\sqrt{1+(\tfrac{2\theta}{\sigma})^2\pi^2(1-\pi)^2}$}.
\]
Although existence of such reflected process may be deduced by standard theory (see, e.g., \cite{Bass} for a general exposition and references therein), we will not dwell here on this issue. In fact in the next section the reflected SDE \eqref{XX}--\eqref{pp} is reduced to an equivalent but simpler one (see \eqref{XX-OP}--\eqref{PH-OP} below) for which a solution can be computed explicitly -- hence implying that \eqref{XX}--\eqref{pp} admits a solution as well.

For $(x,\pi)\in\cO$, let us now consider the problem 
\begin{align}\label{U}
U(x,\pi)=\sup_{\tau}\E_{x,\pi}\left[\exp\left(\int_0^\tau\tfrac{2}{\sigma^2}(\mu_0+\hat \mu\pp_t)dA_t-\rho \tau\right)\right],
\end{align}
where the supremum is taken over all $\P_{x,\pi}$-a.s.~finite stopping times.

Associated with the above problem we also introduce the so-called continuation and stopping sets, denoted by $\cC$ and $\cS$, respectively. These are defined as
\begin{align}
\label{cC} &\cC:=\{(x,\pi)\in[0,+\infty)\times(0,1)\,:\,U(x,\pi)>1\}\\[+3pt]
\label{cS} &\cS:=\{(x,\pi)\in[0,+\infty)\times(0,1)\,:\,U(x,\pi)=1\}
\end{align}
and it is immediate to observe that, if $U=V_x$, then $\cC=\cI$ (recall \eqref{inaction}).

The heuristic that associates \eqref{U} to \eqref{HJBu1}--\eqref{HJBu3} goes as follows: suppose $u\in C^2(\overline\cO)$ is a solution of \eqref{HJBu1}--\eqref{HJBu3} and that 
\[
t\mapsto e^{\int_0^t\frac{2}{\sigma^2}(\mu_0+\hat \mu\pp_s)dA_s-\rho t}u(\XX_t,\pp_t)\quad\text{is $\P$-a.s.~a super-martingale}.
\]
Then $(\cL_{X,\pi}-\rho)u\le 0$ on $\cO$ and an application of Dynkin formula, combined with the use of \eqref{HJBu3} and $u\ge 1$, gives
\begin{align*}
u(x,\pi)\ge&\,\E_{x,\pi}\left[e^{\int_0^\tau\tfrac{2}{\sigma^2}(\mu_0+\hat \mu\pp_t)dA_t-\rho \tau}u(\XX_\tau,\pp_\tau)\right]\\
\ge& \E_{x,\pi}\left[\exp\left(\int_0^\tau\tfrac{2}{\sigma^2}(\mu_0+\hat \mu\pp_t)dA_t-\rho \tau\right)\right]
\end{align*}  
for any stopping time $\tau$. Then $u\ge U$. Moreover, the inequality above becomes a strict equality if we choose $\tau$ as the first exit time from $\cI$ and this concludes the heuristic. 

The rest of this section is devoted to the analysis of problem \eqref{U} in order to show that indeed $U=V_x$ and that $U$ solves \eqref{HJBu1}--\eqref{HJBu3}.


\subsection{A Girsanov transformation}\label{sec:girsanov}

It turns out that the problem may be more conveniently addressed under a different probability measure. As it is customary in problems involving the process $\pi_t$ (see, e.g.,~\cite{EL11}, \cite{K09} or \cite{PJ17}) we introduce here the analogue for $\pp_t$ of the so-called likelihood ratio process
\begin{align}
\PH_t:=\frac{\pp_t}{1-\pp_t},\quad t\ge0.
\end{align}
By direct computation it is not hard to derive the dynamic of $\PH$, for $t>0$, in the form
\begin{align}\label{PH}
\frac{d\PH_t}{\PH_t}=\theta\left(\tfrac{2}{\sigma}dA_t+dW_t+\theta\pp_t dt\right),\qquad \PH_0=\varphi:=\frac{\pi}{1-\pi}.
\end{align}
With the aim of turning $W_t+\theta\int_0^t\pp_s ds$ into a Brownian motion we follow steps as in \cite{EL11} and introduce a new probability measure $\Q$ by its Radon-Nikodym derivative
\begin{align}\label{eta}
\eta_t:=\frac{d\Q}{d\P}\Big|_{\cF_t}=\exp\left(-\int_0^t\theta\pp_sdW_s-\tfrac{1}{2}\int_0^t\theta^2\pp^2_sds\right), \quad t\in[0,T]
\end{align}
for some $T>0$. Under the new measure $\Q$ we have that
\[
W^\Q_t:=W_t+\theta\int_0^t\pp_s ds, \qquad t\in[0,T],
\]
is a Brownian motion and the dynamics of $(\XX,\PH)$ read
\begin{align}
\label{XX-OP}&d\XX_t=\mu_0dt+\sigma dW^\Q_t+dA_t, & \XX_0=x,\\[+4pt]
\label{PH-OP}&d\PH_t=\theta\PH_t(dW^\Q_t+\tfrac{2}{\sigma}dA_t),& \PH_0=\varphi.
\end{align}

One advantage of this formulation is that the process $\XX$ is decoupled from the process $\PH$ and, thanks to \eqref{SK-A}, we see that it is just a Brownian motion with drift $\mu_0$ reflected at zero. In particular this allows to compute a simple expression for $A$. Indeed $\Q_{x,\varphi}$-a.s.~on $[0,T]$ we have (see, \cite[Lemma~3.6.14]{KS})
\begin{align}\label{A}
A_t=x\vee\sup_{0\le s\le t}(-\mu_0s-\sigma W^\Q_s)-x,
\end{align}
Moreover we can express the dynamic for $\PH$~as
\begin{align}\label{PH-OP-dyn}
\PH_t=\varphi \exp\left(\theta W^\Q_t-\tfrac{\theta^2}{2}t +\tfrac{2\theta}{\sigma} A_t\right),\quad\text{$\Q_{x,\varphi}$-a.s.,}
\end{align}
where the dependence on $x$ is given explicitly by \eqref{A}.
Sometimes we will also use the notation $(\XX^x,A^x,\PH^{x,\varphi})$ to express the dependence of $(\XX,A,\PH)$ on the initial point $(x,\varphi)$.

In order to rewrite problem \eqref{U} in the new variables we introduce the process
\begin{align}\label{Z}
Z_t:=\frac{1+\PH_t}{1+\varphi},\qquad\text{$\P_{x,\varphi}$-a.s.}
\end{align}
and notice that $\P_{x,\varphi}(Z_0=1)=1$ and, under the measure $\P_{x,\varphi}$, we have 
\begin{align*}
\frac{dZ_t}{Z_t}=\theta\pp_t\left(\frac{2}{\sigma}dA_t+dW_t+\theta\pp_tdt\right),\qquad t>0.
\end{align*}
Recalling \eqref{eta} and rewriting the above SDE in terms of an exponential gives
\begin{align}\label{Z2}
Z_t=\frac{1}{\eta_t}\exp\left(\int_0^t\tfrac{2\theta}{\sigma}\pp_s dA_s\right),\quad t\in[0,T]
\end{align}
with the same $T>0$ as in \eqref{eta}.

Now, for any $\tau$ and any $(x,\varphi)$ we get
\begin{align}\label{Girs}
\E_{x,\pi}&\,\left[\exp\left(\int_0^{\tau\wedge T}\tfrac{2}{\sigma^2}(\mu_0+\hat \mu\pp_t)dA_t-\rho (\tau\wedge T)\right)\right]\\
=&\,\E_{x,\pi}\left[\exp\left(\tfrac{2\mu_0}{\sigma^2}A_{\tau\wedge T}-\rho(\tau\wedge T) \right)\exp\left(\int_0^{\tau\wedge T}\tfrac{2\theta}{\sigma}\pp_t dA_t\right)\frac{\eta_{\tau\wedge T}}{\eta_{\tau\wedge T}}\right]\nonumber\\
=&\,\E_{x,\pi}\left[\exp\left(\tfrac{2\mu_0}{\sigma^2}A_{\tau\wedge T}-\rho(\tau\wedge T) \right)Z_{\tau\wedge T}\cdot\eta_{\tau\wedge T}\right]\nonumber\\
=&(1+\varphi)^{-1}\EQ_{x,\varphi}\left[\exp\left(\tfrac{2\mu_0}{\sigma^2}A_{\tau\wedge T}-\rho(\tau\wedge T) \right)(1+\PH_{\tau\wedge T})\right].\nonumber
\end{align}
Defining, for all $(x,\varphi)\in\R_+\times\R_+$, the problems
\begin{align*}
&U(x,\pi; T):=\sup_{\tau}\E_{x,\pi}\left[\exp\left(\int_0^{\tau\wedge T}\tfrac{2}{\sigma^2}(\mu_0+\hat \mu\pp_t)dA_t-\rho (\tau\wedge T)\right)\right],\\
&U^\Q(x,\varphi; T):=\sup_{\tau}\EQ_{x,\varphi}\left[\exp\left(\tfrac{2\mu_0}{\sigma^2}A_{\tau\wedge T}-\rho(\tau\wedge T) \right)(1+\PH_{\tau\wedge T})\right],
\end{align*}
we immediately see that \eqref{Girs} implies
\begin{align}\label{eq:T}
U^\Q(x,\varphi; T)=(1+\varphi)U(x,\varphi/(1+\varphi); T).
\end{align}
We would like to extend this equality to the case $T=+\infty$ and this requires a short digression as Girsanov theorem does not directly apply.

Since we are interested in properties of the value functions, here we can define a new probability space $(\Omega',\cF',\OP)$ equipped with a Brownian motion $\OW$ and a filtration $(\cF_t')_{t\ge 0}$, and let $(\XX',\PH')$ be the unique strong solution of the SDE \eqref{XX-OP}--\eqref{PH-OP} driven by $\OW$ (instead of $W^\Q$) with a corresponding process $A'$ as in \eqref{A}. In this setting we can define the stopping problems 
\begin{align*}
\OU(x,\varphi;T):=&\sup_{\tau}\OE_{x,\varphi}\left[\exp\left(\tfrac{2\mu_0}{\sigma^2}A'_{\tau\wedge T}-\rho(\tau\wedge T) \right)(1+\PH'_{\tau\wedge T})\right],\\[+4pt]
\OU(x,\varphi):=&\sup_{\tau}\OE_{x,\varphi}\left[\exp\left(\tfrac{2\mu_0}{\sigma^2}A'_{\tau}-\rho\tau \right)(1+\PH'_{\tau})\right],
\end{align*}
where $\OE$ is the expectation under $\OP$. Now, $U^\Q(x,\varphi;T)=\OU(x,\varphi;T)$ by the equivalence in law of the process $(\XX,\PH,A,W^\Q)$, under $\Q$, and $(\XX',\PH',A',\OW)$, under $\OP$, on $[0,T]$. Further, if we show that 
\begin{align}\label{limits}
\lim_{T\to\infty}\OU(x,\varphi;T)=\OU(x,\varphi)\qquad\text{and}\qquad \lim_{T\to\infty}U(x,\varphi;T)=U(x,\varphi),
\end{align}
then combining these facts with \eqref{eq:T} we obtain
\begin{align}\label{U-OU}
\OU(x,\varphi)=&\lim_{T\to\infty}\OU(x,\varphi;T)=\lim_{T\to\infty}U^\Q(x,\varphi;T)\\
=&(1+\varphi)\lim_{T\to\infty}U(x,\varphi/(1+\varphi); T)=(1+\varphi)U(x,\varphi/(1+\varphi)).\notag
\end{align}
The proof of \eqref{limits} is the same as that of \eqref{OUn} below and we omit it here for brevity.

Finally, with a slight abuse of notation we relabel $(\XX',\PH',A',\OW)=(\XX,\PH,A,\OW)$ and $(\Omega',\cF',(\cF'_t)_{t\ge 0},\OP)=(\Omega,\cF,(\cF_t)_{t\ge 0},\OP)$, so that 
\begin{align}
\label{OU}&\OU(x,\varphi)=\sup_{\tau}\OE_{x,\varphi}\left[\exp\left(\tfrac{2\mu_0}{\sigma^2}A_{\tau}-\rho\tau \right)(1+\PH_{\tau})\right].
\end{align}

Problem \eqref{OU} is somewhat easier to analyse than the original \eqref{U} because the dynamics \eqref{XX-OP}-\eqref{PH-OP} for $(\XX,\PH)$, driven by $\OW$ under $\OP_{x,\varphi}$, are more explicit than the ones of $(\XX,\pp)$, driven by $W$ under $\P_{x,\pi}$ (see \eqref{XX}--\eqref{pp}). 


It is clear from \eqref{U-OU} that $\cC$ and $\cS$ in \eqref{cC}--\eqref{cS} now read
\begin{align}
\label{cC2}&\cC=\{(x,\varphi)\in[0,+\infty)\times(0,+\infty)\,:\,\OU(x,\varphi)>1+\varphi\}\\[+3pt]
\label{cS2}&\cS=\{(x,\varphi)\in[0,+\infty)\times(0,+\infty)\,:\,\OU(x,\varphi)=1+\varphi\}.
\end{align}
\begin{remark}\label{rem:state-space}
The choice $\varphi=0$ corresponds to full information on the drift of \eqref{X0} (i.e.~$\mu=\mu_0$), in which case there is no dynamic for $\PH$. Since problem \eqref{P1} has a well known explicit solution in that setting, and given that $\P_{x,\varphi}(\PH_t>0)=1$ for all $t\ge0$ and any $(x,\varphi)\in[0,+\infty)\times(0,+\infty)$, we will not include $[0,+\infty)\times\{0\}$ in our state-space.
\end{remark}

\subsection{Well posedness and initial properties of the stopping problem.}\label{sec:wp}
At this point we start looking at elementary properties of problem \eqref{OU} which guarantee its well posedness. Recall the following known fact (see \cite[Sec.~3.5.C]{KS}): for $\beta>0$ and $S^{\beta,\sigma}_t:=\sup_{0\le s\le t}(-\beta s-\sigma \OW_s)$ we have 
\begin{align}\label{distr}
\OP(S^{\beta,\sigma}_{\infty}>x)=\exp(-\tfrac{2\beta}{\sigma^2}x)\quad\text{for $x>0$}.
\end{align}
For $\alpha>0$, setting $\beta=\alpha+\tfrac{\sigma^2\rho}{2\alpha}$, the use of \eqref{distr} and 
\[
\tfrac{2\alpha}{\sigma^2}\sup_{0\le s\le t}\left(-\alpha s-\sigma\OW_s\right)-\rho t\le \tfrac{2\alpha}{\sigma^2}\sup_{0\le s\le t}\left(-\beta s-\sigma\OW_s\right)
\]
give the following bound: for any stopping time $\tau$
\begin{align}\label{simple}
\OE&\,\left[e^{\tfrac{2\alpha}{\sigma^2}S^{\alpha,\sigma}_\tau-\rho \tau}\right]\le\OE\left[e^{\tfrac{2\alpha}{\sigma^2}S^{\beta,\sigma}_\tau}\right]\le\OE\left[e^{\tfrac{2\alpha}{\sigma^2}S^{\beta,\sigma}_\infty}\right]\\[+4pt]
=&\,\tfrac{2\beta}{\sigma^2}\int_0^\infty e^{\tfrac{2\alpha}{\sigma^2}x}e^{-\tfrac{2\beta}{\sigma^2}x}dx=\tfrac{2\beta}{\sigma^2}\int_0^\infty e^{-\tfrac{\rho}{\alpha}x}dx<+\infty.\nonumber
\end{align}

A great deal of standard results in optimal stopping theory rely on the assumption that 
\begin{align}\label{eq:ui}
\OE_{x,\varphi}\left[\sup_{t\ge 0}\left(e^{\tfrac{2\mu_0}{\sigma^2}A_t-\rho t}(1+\PH_{t})\right)\right]<+\infty.
\end{align}
In particular \eqref{eq:ui} would normally be used to show that  
\begin{align}\label{tau*}
\tau_*:=\inf\{t\ge0\,:\,(\XX_t,\PH_t)\notin\cC\}
\end{align}
is the minimal optimal stopping time for problem \eqref{OU}, whenever $\OP_{x,\varphi}(\tau_*<+\infty)=1$, otherwise it is the minimal optimal Markov time (see \cite{Sh}) (notice also that for problem \eqref{U} we rewrite \eqref{tau*} in terms of $(\XX,\pp)$). Moreover, \eqref{eq:ui} would also guarantee the (super)-martingale property of the discounted value process: the process $(N_t)_{t\ge 0}$ defined as
\[
N_t:=e^{\tfrac{2\mu_0}{\sigma^2}A_t-\rho t}\,\OU(\XX_t,\PH_t)
\]
satifies 
\begin{align}
\label{supermg}&\text{$(N_t)_{t\ge0}$ is a right-continuous supermartingale},\\[+3pt]
\label{mg}&\text{$(N_{t\wedge\tau_*})_{t\ge0}$ is a right-continuous martingale}.
\end{align}

Assumption \ref{eq:ui} may be fulfilled in our setting by choosing $\rho$ sufficiently large in comparison to the coefficients $(\mu_0,\mu_1,\sigma)$. In fact we notice that the process
\[
e^{\tfrac{2\mu_0}{\sigma^2}A_t-\rho t}\PH_{t}=e^{\tfrac{2\mu_1}{\sigma^2}A_t-\rho t+\theta W_t-\frac{\theta^2}{2}t}
\]
is not uniformly integrable in general. As it turns out, by following a slightly different approach we can still achieve \eqref{tau*}--\eqref{mg} but with no other restriction on $\rho$ than $\rho>0$. 
\vspace{+5pt}

For $n\ge 1$, let us denote $\zeta_n:=\inf\{t\ge0\,:\,\PH_t\ge n\}$ and consider the sequence of problems with value function 
\begin{align}\label{OUn}
\OU^n(x,\varphi):=\sup_{\zeta\le\zeta_n}\OE_{x,\varphi}\left[\exp\left(\tfrac{2\mu_0}{\sigma^2}A_\zeta-\rho\zeta \right)(1+\PH_{\zeta})\right].
\end{align}   
It is clear that such truncated problems fulfill condition \eqref{eq:ui}, since the process $(\XX,\PH)$ is stopped at $\zeta_n$. Hence
\begin{align}\label{zeta*}
\zeta^n_*:=\inf\{t\ge0\,:\,\OU^n(\XX_t,\PH_t)=1+\PH_t\}\wedge\zeta_n
\end{align}
is an optimal stopping time for problem \eqref{OUn}. Moreover, the process $(N^n_t)_{t\ge 0}$ defined as
\begin{align}\label{Sn}
N^n_t:=e^{\tfrac{2\mu_0}{\sigma^2}A_t-\rho t}\,\OU^n(\XX_t,\PH_t)
\end{align}
satisfies the analogue of conditions \eqref{supermg}--\eqref{mg} and we obtain the next useful results
\begin{proposition}\label{prop:OU1n}
The sequence $(\OU^n)_{n\ge1}$ is increasing in $n$ with 
\begin{align}\label{limn}
\lim_{n\to\infty}\OU^n(x,\varphi)=\OU(x,\varphi),\qquad\text{for all $(x,\varphi)\in[0,+\infty)\times(0,+\infty)$}.
\end{align}
Moreover, there exists a universal constant $c_1>0$ such that
\begin{align}\label{sublin}
0\le \OU^n(x,\varphi)\le \OU(x,\varphi)\le 1\!+\!c_1\varphi,\quad \text{for all $(x,\varphi)\in[0,+\infty)\!\times\!(0,+\infty)$.}
\end{align}
\end{proposition}
\begin{proof}
Clearly $\OU^n\le \OU$ for all $n$ and the sequence is increasing because the set of admissible stopping times is increasing. For any $\OP_{x,\varphi}$-a.s.~finite stopping time $\tau$, Fatou's lemma gives 
\begin{align}
\OE_{x,\varphi}&\,\left[\exp\left(\tfrac{2\mu_0}{\sigma^2}A_\tau-\rho\tau \right)(1+\PH_{\tau})\right]\\[+4pt]
\le &\,\liminf_{n\to\infty}\OE_{x,\varphi}\left[\exp\left(\tfrac{2\mu_0}{\sigma^2}A_{\tau\wedge\zeta_n}-\rho(\tau\wedge\zeta_n) \right)(1+\PH_{\tau\wedge\zeta_n})\right]\nonumber\\[+4pt]
\le &\,\liminf_{n\to\infty}\OU^n(x,\varphi).\nonumber
\end{align} 
The latter implies $\OU(x,\varphi)\le \liminf_{n\to\infty}\OU^n(x,\varphi)$ and therefore \eqref{limn}. 

Let us now analyse \eqref{sublin}. For any stopping time $\tau$, using \eqref{PH-OP-dyn} we obtain
\begin{align}\label{sublin1}
&\OE_{x,\varphi}\!\left[\exp\!\left(\tfrac{2\mu_0}{\sigma^2}A_{\tau\wedge\zeta_n}\!-\!\rho(\tau\wedge\zeta_n) \right)\!(1\!+\!\PH_{\tau\wedge\zeta_n})\right]\\[+4pt]
&=\OE_{x,\varphi}\!\left[e^{\tfrac{2\mu_0}{\sigma^2}A_{\tau\wedge\zeta_n}-\rho(\tau\wedge\zeta_n)}\right]\!+\!\varphi\OE_{x,\varphi}\!\left[e^{\tfrac{2\mu_1}{\sigma^2}A_{\tau\wedge\zeta_n}-\rho(\tau\wedge\zeta_n)}e^{\theta\OW_{\tau\wedge\zeta_n}-\frac{\theta^2}{2}(\tau\wedge\zeta_n)}\right],\nonumber
\end{align}
and we can study the two terms separately. For the first one, given that $\mu_0< 0$ then the expectation is trivially bounded above by one.

For the second term in \eqref{sublin1} we first change measure using $d\P^\theta=e^{\theta\OW_t-\frac{\theta^2}{2}t}d\OP$ on $\cF_t$, for $t\in [0,\zeta_n]$, and then notice that $W^\theta_t=\OW_t-\theta t$ is a Brownian motion under $\P^\theta$ for $t\in [0,\zeta_n]$, since the Radon-Nikodym derivative is a bounded martingale. This gives
\[
\OE_{x,\varphi}\left[e^{\tfrac{2\mu_1}{\sigma^2}A_{\tau\wedge\zeta_n}-\rho(\tau\wedge\zeta_n)}e^{\theta\OW_{\tau\wedge\zeta_n}-\frac{\theta^2}{2}(\tau\wedge\zeta_n)}\right]=\E^\theta_{x,\varphi}\left[e^{\tfrac{2\mu_1}{\sigma^2}A_{\tau\wedge\zeta_n}-\rho(\tau\wedge\zeta_n)}\right]\le c_1,
\]
where the final inequality uses \eqref{simple} with $\alpha=\mu_1$ (notice also that $\sup_{0\le s\le t}(-\mu_0 s-\sigma \OW_s)=\sup_{0\le s\le t}(-\mu_1 s-\sigma W^\theta_s)$) and $c_1>0$ is only depending on $(\mu_0,\mu_1,\sigma,\rho)$.

Hence, $\OU^n$ fulfils \eqref{sublin} for all $n\ge 1$ and then \eqref{limn} implies that the bound holds for $\OU$ as well. 
\end{proof}

It is also useful to state a continuity result for $\OU^n$.
\begin{proposition}\label{prop:OUn-c}
For any $n\ge 1$ we have $\OU^n\in C([0,+\infty)\times(0,+\infty))$. Moreover, there exists a universal constant $c>0$ such that for any couple of points $(x_1,\varphi_1)$ and $(x_2,\varphi_2)$ in $[0,+\infty)\times(0,+\infty)$, with $\varphi_2>\varphi_1$, we have
\begin{align}\label{LipOUn}
|\OU^n(x_1,\varphi_1)-\OU^n(x_2,\varphi_2)|\le c\big[(1+\varphi_2)|x_1-x_2|+(\varphi_2-\varphi_1)\big].
\end{align}
Finally, $\varphi\mapsto \OU^n(x,\varphi)$ is non-decreasing for all $x\in[0,+\infty)$.
\end{proposition}
\begin{proof}
Take $x_1<x_2$ and $\varphi\in(0,+\infty)$. Let $\zeta_1=\zeta^n_*(x_1,\varphi)$ be optimal for $\OU^n(x_1,\varphi)$, then by direct comparison
\begin{align*}
\OU^n&\,(x_1,\varphi)-\OU^n(x_2,\varphi)\\[+4pt]
\le&\, \OE\left[e^{\frac{2\mu_0}{\sigma^2}A^{x_1}_{\zeta_1}-\rho\zeta_1}-e^{\frac{2\mu_0}{\sigma^2}A^{x_2}_{\zeta_1}-\rho\zeta_1}\right]
+\varphi\E^\theta\left[e^{\frac{2\mu_1}{\sigma^2}A^{x_1}_{\zeta_1}-\rho\zeta_1}-e^{\frac{2\mu_1}{\sigma^2}A^{x_2}_{\zeta_1}-\rho\zeta_1}\right],\nonumber
\end{align*}  
where, as in the proof of Proposition \ref{prop:OU1n}, we have used $d\P^\theta=e^{\theta\OW_t-\frac{\theta^2}{2}t}d\OP$ to change measure. Next we use that $0\le A^{x_1}- A^{x_2}\le x_2-x_1$ and \eqref{simple} to conclude that 
\begin{align*}
&\E^\theta\left[e^{\frac{2\mu_1}{\sigma^2}A^{x_1}_{\zeta_1}-\rho\zeta_1}-e^{\frac{2\mu_1}{\sigma^2}A^{x_2}_{\zeta_1}-\rho\zeta_1}\right]\le (x_2-x_1)\E^\theta\left[e^{\frac{2\mu_1}{\sigma^2}A^{x_1}_{\zeta_1}-\rho\zeta_1}\right]\le c_1(x_2-x_1)\\[+4pt]
&\OE\left[e^{\frac{2\mu_0}{\sigma^2}A^{x_1}_{\zeta_1}-\rho\zeta_1}-e^{\frac{2\mu_0}{\sigma^2}A^{x_2}_{\zeta_1}-\rho\zeta_1}\right]\le \left|\tfrac{2\mu_0}{\sigma^2}\right|(x_2-x_1).
\end{align*}
Therefore we have $\OU^n\,(x_1,\varphi)-\OU^n(x_2,\varphi)\le c(1+\varphi)(x_2-x_1)$ for $c=c_1\vee |2\mu_0/\sigma^2|$. Symmetric arguments allow to prove the reverse inequality.

Let us now fix $x\in[0,+\infty)$ and $\varphi_1<\varphi_2$ in $(0,+\infty)$. Denote $\zeta^{\varphi_i}_n=\inf\{t\ge0\,:\,\PH^{x,\varphi_i}\ge n\}$ for $i=1,2$ and let $\zeta_i=\zeta_*^n(x,\varphi_i)$ be optimal for $\OU^n(x,\varphi_i)$. Notice first that since $\zeta_2\le \zeta_n^{\varphi_2}\le \zeta_n^{\varphi_1}$ then $\zeta_2$ is admissible for $\OU^n(x,\varphi_1)$. Then, using the same arguments as above we get
\begin{align*}
\OU^n(x,\varphi_2)-\OU^n(x,\varphi_1)\le (\varphi_2-\varphi_1)\E^\theta\left[e^{\frac{2\mu_1}{\sigma^2}A^{x}_{\zeta_2}-\rho\zeta_2}\right]\le c\,(\varphi_2-\varphi_1).
\end{align*}
For the reverse inequality we notice that given any stopping time $\zeta$ then $\zeta\wedge\zeta^{\varphi_2}_n$ is admissible for $\OU^n(x,\varphi_2)$. Using that $\PH^{x,\varphi_1}_{\zeta}\le n$ for $\zeta\le \zeta_n^{\varphi_1}$ and $\PH^{x,\varphi_2}_{\zeta^{\varphi_2}_n}=n$, we get
\begin{align}\label{OUnc2}
\OU^n&\,(x,\varphi_1)-\OU^n(x,\varphi_2)\\[+4pt]
\le &\,\sup_{\zeta\le \zeta_n^{\varphi_1}}\bigg\{(1+n)\E\left[\mathds{1}_{\{\zeta>\zeta^{\varphi_2}_n\}}\left(e^{\frac{2\mu_0}{\sigma^2}A_\zeta-\rho\zeta}-e^{\frac{2\mu_0}{\sigma^2}A_{\zeta^{\varphi_2}_n}-\rho\zeta^{\varphi_2}_n}\right)\right]\nonumber\\[+4pt]
&\hspace{+3pc}+(\varphi_1-\varphi_2)\E\left[\mathds{1}_{\{\zeta\le \zeta^{\varphi_2}_n\}}e^{\frac{2\mu_0}{\sigma^2}A_\zeta-\rho\zeta}\PH^{x,1}_{\zeta}\right]\bigg\}\le 0, \nonumber
\end{align}
where the last inequality also uses that $t\mapsto\frac{2\mu_0}{\sigma^2}A_t-\rho t$ is decreasing.

The above estimates imply \eqref{LipOUn}, and \eqref{OUnc2} also implies monotonicity in $\varphi$.
\end{proof}

We can now state some properties of $\OU$.
\begin{proposition}\label{prop:OU1}
The value function $\OU$ of \eqref{OU} enjoys the properties below:
\begin{itemize}
\item[(i)] $\OU\in C([0,+\infty)\times(0,+\infty))$ and there exists a universal constant $c>0$ such that
\begin{align}\label{Lip-OU}
|\OU(x_1,\varphi_1)-\OU(x_2,\varphi_2)|\le c\left[(1+\varphi_2)|x_2-x_1|+(\varphi_2-\varphi_1)\right]
\end{align}
for all $x_1,x_2\in[0+\infty)$ and $0<\varphi_1<\varphi_2$;
\item[(ii)] $\varphi\mapsto \OU(x,\varphi)$ is convex and non-decreasing for any $x\in\R_+$; 
\item[(iii)] we have $\lim_{\varphi \to 0}\OU(x,\varphi)=1$;
\item[(iv)] the following transversality condition holds
\begin{align}\label{tran}
\lim_{t\to\infty}\OE_{x,\varphi}\left[e^{\frac{2\mu_0}{\sigma^2}A_t-\rho t}\OU(\XX_t,\PH_t)\right]=0,\quad\text{for all $(x,\varphi)\in[0,+\infty)\times(0,+\infty)$.}
\end{align} 
\end{itemize}
\end{proposition}
\begin{proof}
In order to prove $(i)$ it is enough to take $n\to\infty$ in \eqref{LipOUn} and use \eqref{limn}. Let us now show $(ii)$.

Thanks to \eqref{PH-OP-dyn} we know that the map
\begin{align}\label{eq:gain1}
\varphi\mapsto \exp\left(\tfrac{2\mu_0}{\sigma^2}A^x_\tau-\rho\tau \right)(1+\PH^{x,\varphi}_{\tau})
\end{align}
is $\OP$-a.s.~linear for any stopping time $\tau$. Using this fact and the inequality $\sup(f+g)\le \sup(f)+\sup(g)$ it is not hard to verify
\begin{align}
\OU(x,\alpha\varphi_1+(1-\alpha)\varphi_2)\le \alpha\OU(x,\varphi_1)+(1-\alpha)\OU(x,\varphi_2)
\end{align}
for $\alpha\in(0,1)$, $\varphi_1,\varphi_2\in\R_+$ and each given $x\in\R_+$. 
Since the map \eqref{eq:gain1} is monotonic increasing, it also follows that $\varphi\mapsto \OU(x,\varphi)$ is increasing as claimed. (The latter could have also been deduced by monotonicity of $\varphi\mapsto \OU^n(x,\varphi)$.)

Next, we observe that $(iii)$ follows immediately by \eqref{sublin} upon noticing also that $\OU(x,\varphi)\ge 1+\varphi$. 
It only remains to prove $(iv)$. From \eqref{sublin}, and using \eqref{PH-OP-dyn} and 
\[
d\P^\theta=e^{\theta\OW_t-\frac{\theta^2}{2}t}d\OP
\]
we have
\begin{align}
\OE_{x,\varphi}\left[e^{\frac{2\mu_0}{\sigma^2}A_t-\rho t}\OU(\XX_t,\PH_t)\right]\le &\,\OE_{x}\left[e^{\frac{2\mu_0}{\sigma^2}A_t-\rho t}\right]+c_1\varphi \E^\theta_{x}\left[e^{\frac{2\mu_1}{\sigma^2}A_t-\rho t}\right]\\[+4pt]
\le &\, e^{-\frac{\rho}{2}t}+c_1\varphi e^{-\frac{\rho}{2}t}\,\E^\theta\left[e^{\frac{2\mu_1}{\sigma^2}\sup_{0\le s\le t}(-\mu_1s-\sigma W^\theta_s)-\frac{\rho}{2} t}\right]\nonumber,
\end{align}
where we recall that $W^\theta$ is $\P^\theta$-Brownian motion. Using now \eqref{simple} we can find a universal constant $c'_1>0$ such that 
\begin{align}\label{tran-b}
\OE_{x,\varphi}\left[e^{\frac{2\mu_0}{\sigma^2}A_t-\rho t}\OU(\XX_t,\PH_t)\right]\le e^{-\frac{\rho}{2}t}(1+c'_1\varphi).
\end{align}
Then \eqref{tran} follows by taking $t\to\infty$.
\end{proof}

There are several conclusions that one can draw from Proposition \ref{prop:OU1}. First we notice that $(\OU-\OU^n)_{n\ge 1}$ is a decreasing sequence of continuous functions that converges to zero, therefore Dini's theorem implies 
\begin{align}\label{unif}
\lim_{n\to\infty}\sup_{(x,\varphi)\in K}|\OU^n(x,\varphi)-\OU(x,\varphi)|=0,
\end{align}
for any compact $K\subset[0,+\infty)\times(0,+\infty)$.
Now we can use this fact and an argument inspired by \cite[Lem.~4.17]{ChDeA16} and \cite[Lem.~6.2]{ChDeA15} to prove the next lemma.
\begin{lemma}\label{lem:st-c}
The sequence $(\zeta_*^n)_{n\ge 1}$ (see \eqref{zeta*}) is increasing in $n$ and for all $(x,\varphi)\in[0,+\infty)\times(0,+\infty)$~we have
\begin{align}\label{st-c}
\OP_{x,\varphi}\left(\lim_{n\to\infty}\zeta_*^n=\tau_*\right)=1
\end{align}
with $\tau_*$ as in \eqref{tau*}.
\end{lemma}
\begin{proof}
Since $\OU^n$ is increasing in $n$, it is clear that the sequence $(\zeta_*^n)_{n\ge 1}$ defined by \eqref{zeta*} is also increasing and $\zeta_*^n\le \tau_*$ for all $n\ge 1$, $\OP_{x,\varphi}$-a.s. For $(x,\varphi)\in \cS$ it is clear that \eqref{st-c} holds. For fixed $(x_0,\varphi_0)\in\cC$ we argue by contradiction and assume that 
\begin{align}
\OP_{x_0,\varphi_0}\left(\lim_{n\to\infty}\zeta_*^n<\tau_*\right)>0.
\end{align}
Letting $\Omega_0:=\{\omega\,:\,\lim_{n\to\infty}\zeta_*^n<\tau_*\}$ we pick an arbitrary $\omega\in\Omega_0$. Then there is $\delta_\omega>0$ such that $\tau_*(\omega)>\delta_\omega$. This implies that there also exists $c_\omega>0$ such that 
\begin{align}\label{st-c1}
\inf_{t\in[0,\delta_\omega]}\left(\OU(\XX_t,\PH_t)-(1+\PH_t)\right)(\omega)> c_\omega,
\end{align}
thanks to $(i)$ in Proposition \ref{prop:OU1} and because the process $t\mapsto(\XX_t,\PH_t)$ is continuous up to a null subset of $\Omega$. Then the image of $(\XX_t,\PH_t)(\omega)$ for $t\in[0,\delta_\omega]$ is a compact that we denote by $K_{\omega,\delta}$ and \eqref{st-c1} is equivalent to
\begin{align}\label{st-c2}
\inf_{(x,\varphi)\in K_{\omega,\delta}}\left(\OU(x,\varphi)-(1+\varphi)\right)> c_\omega.
\end{align}
Thanks to \eqref{unif} we can find $N_{\omega,\delta}\ge 1$ such that \eqref{st-c2} holds with $\OU^n$ instead of $\OU$, for all $n\ge N_{\omega,\delta}$. This implies 
$\lim_{n\to\infty}\zeta_*^n(\omega)\ge \delta_\omega$.

The argument may be repeated for any $\delta_\omega<\tau_*(\omega)$ and therefore $\lim_{n\to\infty}\zeta_*^n(\omega)\ge \tau_*(\omega)$, hence a contradiction with the definition of $\Omega_0$. 
\end{proof}

The above lemma implies optimality of $\tau_*$ as explained in the next proposition.
\begin{proposition}\label{prop:opt}
The stopping time $\tau_*$ in \eqref{tau*} is optimal for problem \eqref{OU} in the sense that for all $(x,\varphi)\in[0,+\infty)\times(0,+\infty)$ we have
\begin{align}\label{opt-1}
\OU(x,\varphi)=\OE_{x,\varphi}\left[e^{\frac{2\mu_0}{\sigma^2}A_{\tau_*}-\rho\tau_*}\big(1+\PH_{\tau_*}\big)\mathds{1}_{\{\tau_*<+\infty\}}\right].
\end{align}
Moreover the (super)-martingale properties \eqref{supermg}--\eqref{mg} hold.
\end{proposition}
\begin{proof}
We start by showing \eqref{supermg}--\eqref{mg}. Recall the process $(N^n_t)_{t\ge 0}$ defined in \eqref{Sn} and notice that \eqref{supermg}--\eqref{mg} hold for such process. Then, for any $s\ge t$ we have $\OP_{x,\varphi}$-a.s.
\begin{align*}
e^{\frac{2\mu_0}{\sigma^2}A_{t\wedge\zeta_n}-\rho (t\wedge\zeta_n)}\,\OU^n(\XX_{t\wedge\zeta_n},\PH_{t\wedge\zeta_n})\ge&\,\OE_{x,\varphi}\left[e^{\frac{2\mu_0}{\sigma^2}A_{s\wedge\zeta_n}-\rho (s\wedge\zeta_n)}\,\OU^n(\XX_{s\wedge\zeta_n},\PH_{s\wedge\zeta_n})\Big|\cF_t\right].
\end{align*}
Letting $n\to\infty$, dominated convergence and \eqref{limn} imply that \eqref{supermg} holds. Similarly we have 
\begin{align*}
&e^{\frac{2\mu_0}{\sigma^2}A_{t\wedge\zeta^n_*}-\rho (t\wedge\zeta^n_*)}\,\OU^n(\XX_{t\wedge\zeta^n_*},\PH_{t\wedge\zeta^n_*})\\[+4pt]
&\,=\OE_{x,\varphi}\left[e^{\frac{2\mu_0}{\sigma^2}A_{s\wedge\zeta^n_*}-\rho (s\wedge\zeta^n_*)}\,\OU^n(\XX_{s\wedge\zeta^n_*},\PH_{s\wedge\zeta^n_*})\Big|\cF_{t}\right],\quad\OP_{x,\varphi}-a.s.
\end{align*}
Then, taking $n\to\infty$ and using dominated convergence \eqref{unif} and \eqref{st-c} we obtain that \eqref{mg} holds too.

In order to prove \eqref{opt-1}, we notice that \eqref{mg} implies, for any $t\ge 0$ 
\begin{align}
\OU(x,\varphi)=&\,\OE_{x,\varphi}\left[e^{\frac{2\mu_0}{\sigma^2}A_{t\wedge\tau_*}-\rho (t\wedge\tau_*)}\,\OU(\XX_{t\wedge\tau_*},\PH_{t\wedge\tau_*})\right]\\[+4pt]
=&\OE_{x,\varphi}\left[e^{\frac{2\mu_0}{\sigma^2}A_{\tau_*}-\rho \tau_*}\,(1+\PH_{\tau_*})\mathds{1}_{\{\tau_*\le t\}}+e^{\frac{2\mu_0}{\sigma^2}A_{t}-\rho t}\,\OU(\XX_{t},\PH_{t})\mathds{1}_{\{\tau_*>t\}}\right]\nonumber
\end{align}
where we have used continuity of $\OU$ in the second equality. Letting $t\to\infty$, the transversality condition \eqref{tran} gives \eqref{opt-1}.
\end{proof}
Before closing this section we illustrate consequences of Proposition \ref{prop:OU1} for the shape of the continuation and stopping sets $\cC$ and $\cS$. These are summarised in the next corollary. \begin{corollary}\label{cor:CS1}
The continuation set $\cC$ is open and the stopping set $\cS$ is closed. The continuation set is connected in the $\varphi$ variable, i.e., for all $\varphi'>\varphi$ we have
\begin{align}\label{eq:conn}
(x,\varphi)\in \cC\implies (x,\varphi')\in\cC.
\end{align}
\end{corollary}
\begin{proof}
The first statement is trivial due to $(i)$ in Proposition \ref{prop:OU1}. The second statement follows from the fact that
$\varphi\mapsto\OU(x,\varphi)-(1+\varphi)$ is convex due to $(ii)$ in Proposition \ref{prop:OU1}, it is non-negative and $(iii)$ in Proposition \ref{prop:OU1} holds. 
\end{proof}

For future frequent use we define
\begin{align}
\label{psi}\psi(x):=\sup\{\varphi\in(0,+\infty)\,:\,(x,\varphi)\in\cS\},
\end{align}
for any $x\in[0,+\infty)$, with the convention that $\sup\varnothing=0$. Clearly $\cC$ and $\psi$ are related by
\begin{align}\label{cC-psi}
\cC=\{(x,\varphi)\in[0,+\infty)\times(0,+\infty)\,:\,\varphi> \psi(x)\}
\end{align}
(see also Remark \ref{rem:state-space}).

Next we will infer monotonicity of $\psi(\cdot)$ and therefore the existence of a generalised inverse $c(\cdot)$, which is more convenient for a fuller geometric characterisation of $\cC$. This will be done in the next sections.

\subsection{A parabolic formulation}\label{sec:parab}

Since the process $(\XX,\PH)$ is driven by the same Brownian motion we can equivalently consider a 2-dimensional state dynamic in which only one component has a diffusive part. This is done according to a method similar to the one used in several papers addressing partial information, including \cite{DeAGV17, PJ17}.

Let us define a new process $(\YY_t)_{t\ge0}$ by setting, $\OP_{x,\varphi}$-a.s.~for all $t\ge0$
\begin{align}\label{YY}
\YY_t:=\frac{\sigma}{\theta}\ln(\PH_t)-\XX_t.
\end{align}
Then, letting $y:=\frac{\sigma}{\theta}\ln(\varphi)-x$, it is easy to verify that the couple $(\XX,\YY)$ evolves under $\OP_{x,y}$ according to
\begin{align}
\label{XX2}&d\XX_t=\mu_0dt+\sigma d\OW_t+dA_t, & \XX_0=x,\\[+4pt]
\label{YY2}&d\YY_t=-\tfrac{1}{2}(\mu_1+\mu_0)dt+dA_t,& \YY_0=y.
\end{align}
In order to rewrite our problem \eqref{OU} in terms of the new dynamics we set 
\begin{align}\label{OU-UU}
\UU(x,y):=\OU(x,\exp\tfrac{\theta}{\sigma}(x+y)),\qquad (x,y)\in[0,+\infty)\times\R
\end{align}
and from \eqref{OU} we obtain
\begin{align}\label{UU}
\UU(x,y)=\sup_{\tau}\OE_{x,y}\left[\exp\left(\tfrac{2\mu_0}{\sigma^2}A_\tau-\rho\tau\right)\left(1+\exp\big[\tfrac{\theta}{\sigma}(\XX_\tau+\YY_\tau)\big]\right)\right].
\end{align}
It is convenient in what follows to set
\begin{align}\label{g}
g(x,y):=1+\exp\big[\tfrac{\theta}{\sigma}(x+y)\big],\quad(x,y)\in[0,+\infty)\times\R
\end{align}
and notice that 
\[
\cC=\{(x,y)\in[0,+\infty)\times \R:\UU(x,y)>g(x,y)\}.
\]

Another formulation of the problem, which will be useful below, may be obtained by an application of Dynkin's formula (up to standard localisation arguments). Indeed we can write  
\begin{align}\label{uu}
\uu(x,y):=&\UU(x,y)-g(x,y)\nonumber\\[+4pt]
=&\sup_{\tau}\OE_{x,y}\bigg[\int_0^\tau e^{\frac{2\mu_0}{\sigma^2}A_t-\rho t}\,2\sigma^{-2}\left(\mu_0+\mu_1e^{\frac{\theta}{\sigma}\YY_t}\right)\!dA_t\\[+4pt]
&\hspace{+50pt}-\rho\int_0^\tau e^{\frac{2\mu_0}{\sigma^2}A_t-\rho t}g(\XX_t,\YY_t)dt\bigg],\nonumber
\end{align}
where we have also used that $dA_t=\mathds{1}_{\{\XX_t=0\}}dA_t$ (cf.~\eqref{SK-A}). 
Recalling from Proposition \ref{prop:OU1} that $\varphi\mapsto \OU(x,\varphi)-(1+\varphi)$ is convex and non-negative with $\OU(x,0+)=1$, it follows that the mapping is also non-decreasing. Then we have 
\begin{align}\label{uhaty}
\text{$y\mapsto \uu(x,y)$ is non-decreasing}.
\end{align}

For frequent future use we introduce the second order operator $\cL_{X,Y}$ associated to $(\XX,\YY)$. That is, for $f\in C^{1,2}([0,+\infty)\times\R)$ and $(x,y)\in[0,+\infty)\times\R$ we set
\begin{align}\label{LXY}
(\cL_{X,Y}f)(x,y):=\big(-\tfrac{1}{2}(\mu_1+\mu_0)f_y+\tfrac{\sigma^2}{2}f_{xx}+\mu_0 f_x\big)(x,y).
\end{align}

Notice that $(i)$ in Proposition \ref{prop:OU1} implies that $\UU$ and $\uu$ are both continuous on $[0,+\infty)\times\R$. Thanks to the parabolic formulation and recalling the martingale property \eqref{mg}, we can rely on standard optimal stopping theory and classical PDE results to state the next lemma (see, e.g., \cite[Sec.~2.7, Thm.~7.7]{KS2}). 
\begin{lemma}\label{lem:BVP1}
Given any open set $\mathcal{R}$ whose closure is contained in $\cC$, the function $\UU$ is the unique classical solution of the boundary value problem
\begin{align}\label{BVP1}
(\cL_{X,Y}-\rho) f=0,\:\:\:\text{in $\mathcal{R}$ and $f|_{\partial\mathcal{R}}=\UU|_{\partial\mathcal{R}}$}.
\end{align}
Hence $\UU$ is $C^{1,2}$ in $\cC\cap((0,+\infty)\times\R)$. 
\end{lemma}

 
Now we turn to the analysis of the geometry of $\cC$. First we show that $\cC\neq\varnothing$. 
\begin{proposition}\label{prop:notempty}
We have $\cC\neq\varnothing$ and in particular $\{0\}\times(y_\ell,+\infty)\subset \cC$
with $y_\ell:=\tfrac{\sigma}{\theta}\ln(-\tfrac{\mu_0}{\mu_1})$.
\end{proposition}
\begin{proof}
Fix $\eps>0$, take $y>y_\ell+\eps$ and let 
\[
\tau_\ell:=\inf\{t\ge0\,:\,(\XX_t,\YY_t, A_t)\notin[0,1)\times(y_\ell\!+\!\eps,\infty)\times[0,1)\}.
\] 
Notice that there exists $c_{1,\eps}>0$, $c_{2,\eps}>0$ such that $\OP_{0,y}$-a.s.
\begin{align*}
g(\XX_t,\YY_t)\le c_{2,\eps}\quad\text{and}\quad \mu_0+\mu_1e^{\frac{\theta}{\sigma}\YY_t}\ge c_{1,\eps}
\end{align*}
for all $t\in[0,1\wedge\tau_\ell]$, given that $y_\ell+\eps\le \YY_{t\wedge \tau_\ell}\le y+\tfrac{1}{2}|\mu_0+\mu_1|+1$.
Then, recalling \eqref{uu} and that 
\[
A_t=\sup_{0\le s\le t}(-\mu_0 s-\sigma \OW_s)=S^{\mu_0,\sigma}_t,\quad\text{$\OP_{0,y}$-a.s.~for all $t\ge0$},
\]
we immediately obtain
\begin{align*}
\uu(0,y)\ge&\, \OE_{0,y}\Big[\int_0^{u\wedge\tau_\ell} e^{\frac{2\mu_0}{\sigma^2}A_t-\rho t}\frac{2}{\sigma^2}\left(\mu_0+\mu_1e^{\frac{\theta}{\sigma}\YY_t}\right)dA_t-\rho\!\int_0^{u\wedge\tau_\ell} e^{\frac{2\mu_0}{\sigma^2}A_t-\rho t}g(\XX_t,\YY_t)dt\Big]\nonumber\\
\ge& \OE_{0,y}\left[c'_{1,\eps}S^{\mu_0,\sigma}_{u\wedge\tau_\ell}-c'_{2,\eps}(u\wedge\tau_\ell)\right],
\end{align*}
for some $c'_{1,\eps}>0$, $c'_{2,\eps}>0$ and all $u\in(0,1]$. Next we obtain (cf.~also \cite[Lem.~13]{Pe17})
\begin{align}\label{sup}
\uu(0,y)\ge&\,c'_{1,\eps}\OE_{0,y}\left[\mathds{1}_{\{u\le \tau_\ell\}}S^{\mu_0,\sigma}_{u}\right]-c'_{2,\eps} u\\[+4pt]
\ge&\,c'_{1,\eps}\OE_{0,y}\left[S^{\mu_0,\sigma}_u\right]-c'_{1,\eps}\OE_{0,y}\left[\mathds{1}_{\{\tau_\ell<u\}}S^{\mu_0,\sigma}_u\right]-c'_{2,\eps}u\nonumber\\[+4pt]
\ge&\,\,c'_{1,\eps}\OE_{0,y}\left[S^{\mu_0,\sigma}_u\right]-c'_{1,\eps}\sqrt{\OP_{0,y}(\tau_\ell<u)}\sqrt{\OE_{0,y}\left[(S^{\mu_0,\sigma}_u)^2\right]}-c'_{2,\eps}u.\nonumber
\end{align}
Notice now that for each $u\ge 0$ one has 
\[
\text{Law}(\sup_{0\le s\le u}\OW_s)=\text{Law}(|\OW_u|)=\text{Law}(|\OW_1|\sqrt{u}).
\] 
Then, for some suitable $c>0$ that may vary from line to line but is independent of $u>0$, we obtain
\begin{align}\label{sup0}
\OE_{0,y}\left[(S^{\mu_0,\sigma}_u)^2\right]\le c\,\Big(u^2+\OE\left[\big(\sup_{0\le s\le u}\OW_s\big)^2\right]\Big)=c\,\Big(u^2+u\,\OE\left[|\OW_1|^2\right]\Big).
\end{align}
Moreover, from \eqref{YY2} we observe that if $\mu_0+\mu_1\le 0$ the process $\YY$ will never reach $y_\ell+\eps$, whereas if $\mu_0+\mu_1>0 $ then $\YY_t\le y_\ell\implies t\ge 2(y-y_\ell-\eps)/(\mu_0+\mu_1)=:t_\ell$. Hence, with no loss of generality we may take $u<t_\ell$ and get 
\begin{align}
\OP_{0,y}(\tau_\ell<u)=&\,\OP_{0,y}\left(\sup_{0\le s\le u}\XX_s\ge 1\:\:\text{or}\:\:A_u\ge 1\right)\notag\\
\le&\,\OP_{0,y}\left(\sup_{0\le s\le u}\XX_s\ge 1\right)+\OP_{0,y}\left(A_u\ge 1\right)\notag.
\end{align}
For the first term on the right-hand side above we have
\begin{align}
\OP_{0,y}\left(\sup_{0\le s\le u}\XX_s\ge 1\right)=&\OP\left(\sup_{0\le s\le u}\left[\sup_{0\le v\le s}\big(\mu_0(s-v)+\sigma(\OW_s-\OW_v)\big)\right]\ge 1\right)\nonumber\\
&=\,\OP\left(\sup_{0\le s\le u}(\mu_0s+\sigma\OW_s)\ge 1\right)\nonumber\\
&\le\,\OE\left[\sup_{0\le s\le u}(\mu_0s+\sigma\OW_s)^2\right]\le c\,\big(u^2+u\,\OE\left[|\OW_1|^2\right]\big),\nonumber
\end{align}
where we have used Markov's inequality in the penultimate inequality. It is easy to check that we have the same bound also for $\OP_{0,y}(A_u\ge 1)$ and therefore
\begin{align}\label{sup1}
\OP_{0,y}(\tau_\ell<u)\le 2c\,\big(u^2+u\,\OE\left[|\OW_1|^2\right]\big).
\end{align}
Finally, we also notice that since $\mu_0<0$ we have
\begin{align}\label{sup2}
\OE_{0,y}\left[S^{\mu_0,\sigma}_u\right]\ge \sigma\OE\left[\sup_{0\le s\le u}\OW_s\right]=\sigma\,\sqrt{u}\cdot\OE|\OW_1|.
\end{align}

Plugging \eqref{sup0}--\eqref{sup2} in \eqref{sup} we obtain
\begin{align*}
\uu(0,y)\ge c''_{1,\eps}\sqrt{u}-c''_{2,\eps}\left(u+u^{3/2}+u^2\right),
\end{align*}
with suitable constants $c''_{1,\eps}>0$ and $c''_{2,\eps}>0$.
Then, taking $u$ sufficiently small we obtain $\uu(0,y)>0$ as claimed.
\end{proof}

By \eqref{YY2} we notice that, for $\XX$ away from $0$ the process $\YY$ could either have a positive drift or a negative one. Interestingly, this dichotomy also produces substantially different technical difficulties. Recalling \eqref{cC-psi} we start by observing that 
\[
\varphi>\psi(x)\iff e^{(\theta/\sigma)(x+y)}>\psi(x)\iff y>\chi(x)
\]
where 
\begin{align}\label{xi}
\chi(x):=\tfrac{\sigma}{\theta}(\ln\psi(x)-x),\qquad x\in[0,+\infty).
\end{align} 
Hence we have that \eqref{cC-psi} is equivalent to
\begin{align}\label{cC-y}
\cC=\{(x,y)\in[0,+\infty)\times\R\,:\,y> \chi(x)\}.
\end{align}
Before going further it is convenient to introduce  
\begin{align}
\label{cC-S-y}\cC_y:=\{x\in\R_+\,:\,(x,y)\in\cC\}\quad\text{and}\quad\cS_y:=\{x\in\R_+\,:\,(x,y)\in\cS\},
\end{align}
for any $y\in\R$. The geometry of $\cC$ in the coordinates $(x,y)$ is explained in Proposition \ref{prop:b1} and \ref{prop:b2} below.

\begin{proposition}\label{prop:b1}
Assume $\mu_1+\mu_0\ge 0$, then there exists a unique non decreasing function $b:\R\to[0,+\infty]$ such that $\cS_y=[b(y),+\infty)$ for all $y\in\R$ (with $\cS_y=\varnothing$ if $b(y)=+\infty$).
\end{proposition}
\begin{proof}
First we show that $(x,y)\in\cS\implies(x',y)\in\cS$ for all $x'\ge x$. Fix $(x,y)\in\cS$ and $x'> x$, then we know from \eqref{cC-y} that  $(-\infty,y]\times\{x\}\in\cS$. Due to \eqref{YY2} we have $\YY$ non-increasing, during excursions of $\XX$ away from zero. This implies that the process $(\XX^{x'},\YY^{x',y})$ cannot reach $x=0$ before hitting the half-line $(-\infty,y]\times\{x\}$, thus implying $\OP_{x',y}(\tau_*<\tau_0)=1$ for $\tau_0:=\inf\{t\ge0\,:\,\XX_t=0\}$. Hence \eqref{uu} gives $\uu(x',y)\le 0$ for all $x'\ge x$, as claimed.

Now, for each $y\in\R$ we can define $b(y):=\inf\{x\in[0,+\infty)\,:\,(x,y)\in\cS\}$ and therefore $\cS_y=[b(y),+\infty)$. Combining the latter with \eqref{cC-y} gives that $y\mapsto b(y)$ is non-decreasing. 
\end{proof}

Next we want to show that a result similar to Proposition \ref{prop:b1} also holds for $\mu_1+\mu_2<0$, under a mild additional condition. However, in this case we need first to compute an expression for the derivative $\UU_y$. 
\begin{lemma}\label{lem:UUy}
For all $(x,y)\in((0,+\infty)\times\R)\setminus\partial\cC$ we have
\begin{align}\label{UUy}
\UU_y(x,y)=\OE_{x,y}\left[\frac{\theta}{\sigma}\exp\left(\tfrac{2\mu_0}{\sigma^2}A_{\tau_*}-\rho\tau_*+\tfrac{\theta}{\sigma}(\XX_{\tau_*}+\YY_{\tau_*})\right)\mathds{1}_{\{\tau_*<+\infty\}}\right].
\end{align}
\end{lemma}
\begin{proof}
The claim is trivial if $(x,y)\in\cS\setminus\partial\cC$ since $\OP_{x,y}(\tau_*=0)=1$ therein. Take $(x,y)\in\cC$ and let $\tau:=\tau_*(x,y)$ be optimal for $\UU(x,y)$. Then for $\eps>0$, using \eqref{supermg} and \eqref{mg}, we have
\begin{align*}
\UU&\,(x,y+\eps)-\UU(x,y)\\[+4pt]
\ge& \OE\left[\exp(\tfrac{2\mu_0}{\sigma^2}A^x_{\tau\wedge t}-\rho(\tau\wedge t)\Big(\UU(\XX^x_{\tau\wedge t},\YY^{x,y+\eps}_{\tau\wedge t})-\UU(\XX^x_{\tau\wedge t},\YY^{x,y}_{\tau\wedge t})\Big)\right]\nonumber\\[+4pt]
\ge& \OE\left[\mathds{1}_{\{\tau\le t\}}\exp(\tfrac{2\mu_0}{\sigma^2}A^x_{\tau}-\rho\tau)\Big(g(\XX^x_{\tau},\YY^{x,y+\eps}_{\tau})-g(\XX^x_{\tau},\YY^{x,y}_{\tau})\Big)\right]\nonumber\\[+4pt]
&+\OE\left[\mathds{1}_{\{\tau> t\}}\exp(\tfrac{2\mu_0}{\sigma^2}A^x_{t}-\rho t)\Big(\UU(\XX^x_{t},\YY^{x,y+\eps}_{t})-\UU(\XX^x_{t},\YY^{x,y}_{t})\Big)\right].\nonumber
\end{align*}
Recall \eqref{OU-UU}, \eqref{tran} and \eqref{simple}. Then letting $t\to\infty$ and using also dominated convergence gives 
\begin{align}\label{UUy-1}
\UU&\,(x,y+\eps)-\UU(x,y)\\[+4pt]
\ge&\,\OE\left[\exp(\tfrac{2\mu_0}{\sigma^2}A^x_{\tau}-\rho\tau)\Big(g(\XX^x_{\tau},\YY^{x,y+\eps}_{\tau})-g(\XX^x_{\tau},\YY^{x,y}_{\tau})\Big)\mathds{1}_{\{\tau<+\infty\}}\right].\nonumber
\end{align}
The same argument may be applied to obtain
\begin{align}
\label{UUy-2}\UU&\,(x,y)-\UU(x,y-\eps)\\[+4pt]
\le&\, \OE\left[\exp(\tfrac{2\mu_0}{\sigma^2}A^x_{\tau}-\rho\tau)\Big(g(\XX^x_{\tau},\YY^{x,y}_{\tau})-g(\XX^x_{\tau},\YY^{x,y-\eps}_{\tau})\Big)\mathds{1}_{\{\tau< +\infty\}}\right].\nonumber
\end{align}
We divide both expressions by $\eps$ and let $\eps\to0$. Then, recalling that $\UU\in C^{1,2}$ in $\cC$ (Lemma \ref{lem:BVP1}), noticing that $\partial_y Y^y_t=1$ for all $t\ge0$ and that $\tau$ was chosen independent of $\eps$, we obtain \eqref{UUy}.
\end{proof}
\begin{proposition}\label{prop:b2}
Assume $\mu_1+\mu_0< 0$ and $\rho\ge \tfrac{\theta}{2\sigma}|\mu_1+\mu_0|$. Then there exists a unique non decreasing function $b:\R\to[0,+\infty]$ such that $\cS_y=[b(y),+\infty)$ for all $y\in\R$ (with $\cS_y=\varnothing$ if $b(y)=+\infty$).
\end{proposition} 
\begin{proof}
First notice that if $\cS_y=[b(y),+\infty)$ for all $y\in\R$, then $b$ is non-decreasing due to \eqref{cC-y}. Then it remains to prove existence of $b$.

Fix $y\in\R$. Then we have two possibilities:
\begin{itemize}
\item[$(i)$] $\uu_x(x,y)\le 0$ for all $x\in(0,+\infty)$ such that $(x,y)\in\cC$;
\item[$(ii)$] there exists $x_0\in(0,+\infty)$ with $(x_0,y)\in\cC$ and $\uu_x(x_0,y)> 0$.
\end{itemize}

In case $(i)$ there exists a unique point $b(y)\in[0,+\infty]$ such that $\cS_y=[b(y),+\infty)$. 
In case $(ii)$ we argue in two steps. First we show that $(ii)$ implies $[x_0,+\infty)\times\{y\}\in\cC$ and then we show that $[x_0,+\infty)\times\{y\}\in\cC$ leads to a contradiction. Hence only $(i)$ is possible, for all $y\in\R$. 
\vspace{+4pt}

\emph{Step 1}. ($(ii)\implies[x_0,+\infty)\times\{y\}\in\cC$). From Lemma \ref{lem:BVP1} and the definition of $\uu$ in \eqref{uu} we know that 
\begin{align}\label{Lu}
\cL_{X,Y}\uu-\rho\uu=\rho g\quad\text{in $\cC\cap((0,+\infty)\times\R)$,}
\end{align}
where we recall $g$ as in \eqref{g}.
So, in particular at $(x_0,y)$ we have $\mu_0\uu_x(x_0,y)<0$ and
\begin{align}\label{S1-0}
\tfrac{\sigma^2}{2}\uu_{xx}(x_0,y)=&\,\rho g(x_0,y)+\rho\uu(x_0,y)-\mu_0\uu_x(x_0,y)+\tfrac{1}{2}(\mu_0+\mu_1)\uu_y(x_0,y)\\[+4pt]
>&\,\rho g(x_0,y)+\rho\uu(x_0,y)+\tfrac{1}{2}(\mu_0+\mu_1)\uu_y(x_0,y).\nonumber
\end{align}

Next we use the probabilistic representation of $\UU_y$ \eqref{UUy} to find a lower bound for the right hand side of \eqref{S1-0}. In particular, by direct comparison of \eqref{UU} and \eqref{UUy} (recall also \eqref{opt-1}) we obtain
\begin{align}
\UU_y(x,y)=\frac{\theta}{\sigma}\left(\UU(x,y)-\OE_{x,y}\left[\exp\left(\tfrac{2\mu_0}{\sigma^2}A_{\tau_*}-\rho\tau_*\right)\mathds{1}_{\{\tau_*<+\infty\}}\right]\right)
\end{align}
and consequently
\begin{align}\label{S1-1}
\uu_y(x,y)=\frac{\theta}{\sigma}\uu(x,y)+\frac{\theta}{\sigma}\left(1-\OE_{x,y}\left[\exp\left(\tfrac{2\mu_0}{\sigma^2}A_{\tau_*}-\rho\tau_*\right)\mathds{1}_{\{\tau_*<+\infty\}}\right]\right).
\end{align}
Plugging \eqref{S1-1} in the right hand side of \eqref{S1-0} we immediately find
\begin{align}
\tfrac{\sigma^2}{2}\uu_{xx}(x_0,y)>&\,(\rho-\tfrac{\theta}{2\sigma}|\mu_0+\mu_1|)(\uu(x_0,y)+1)+\rho e^{\frac{\theta}{\sigma}(x_0+y)}\nonumber\\[+4pt]
&+\tfrac{\theta}{2\sigma}|\mu_0+\mu_1|\OE_{x_0,y}\left[\exp\left(\tfrac{2\mu_0}{\sigma^2}A_{\tau_*}-\rho\tau_*\right)\mathds{1}_{\{\tau_*<+\infty\}}\right]>0.
\end{align}

The latter implies that $\uu_x(\cdot,y)$ is increasing in a right-neighbourhood of $x_0$. Hence we can repeat the argument for any point in such neighbourhood and eventually we conclude that $\uu_x(\cdot,y)>0$ on $[x_0,+\infty)$. Then it must be $[x_0,+\infty)\times\{y\}\in\cC$. 
\vspace{+4pt}

\emph{Step 2}. ($[x_0,+\infty)\times\{y\}\in\cC$ is impossible). Fix $(x_0,y_0)$ such that $[x_0,+\infty)\times\{y_0\}\in\cC$. Recalling \eqref{cC-y} we then obtain $[x_0,+\infty)\times[y_0,+\infty)\in\cC$ and therefore $\OP_{x,y}(\tau_*=+\infty)=1$ for any $(x,y)\in (x_0,+\infty)\times(y_0,+\infty)$, because $\YY$ is increasing (cf.~\eqref{YY2}). Then for any such $(x,y)$, \eqref{mg} gives
\begin{align*}
\UU(x,y)=\OE_{x,y}\left[e^{\frac{2\mu_0}{\sigma^2}A_t-\rho t}\UU(\XX_t,\YY_t)\right],\qquad\text{for all $t\ge0$}.
\end{align*} 
Letting $t\to \infty$, the transversality condition \eqref{tran} gives the absurd $\UU(x,y)=0$. 
\end{proof}

Combining the above Propositions \ref{prop:b1} and \ref{prop:b2} with \eqref{cC-y} gives the next corollary.
\begin{corollary}\label{cor:chi}
Assume either $\mu_1+\mu_0\ge 0$ or $\mu_1+\mu_0< 0$ with $\rho\ge \tfrac{\theta}{2\sigma}|\mu_1+\mu_0|$. Then the map $x\mapsto\chi(x)$ is non-decreasing.
\end{corollary}
\noindent We can say that $\chi$ is the (generalised) inverse of $b$ in a sense that will be clarified later in Section \ref{sec:refl-cr}.

\section{Fine properties of the value function and of the boundary}\label{sec:fine}

In this section we continue our study of the optimal stopping problem by proving that its value function is $C^1$ and by illustrating properties of the optimal boundary in the different coordinate systems (i.e.~$(x,\pi)$, $(x,\varphi)$ and $(x,y)$).

\subsection{Regularity of value function and optimal boundary}

Combining Propositions \ref{prop:b1} and \ref{prop:b2} we conclude that if either of the two conditions below holds:
\begin{itemize}
\item[$(i)$] $\mu_1+\mu_0\ge 0$,
\item[$(ii)$] $\mu_1+\mu_0< 0$ and $\rho\ge \tfrac{\theta}{2\sigma}|\mu_1+\mu_0|$,
\end{itemize}
then there is a non-decreasing optimal boundary $b$ such that
\begin{align}\label{cS3}
\cS=\{(x,y)\in[0,+\infty)\times\R\,:\,x\ge b(y)\}.
\end{align}
Since we will only consider cases $(i)-(ii)$ in the rest of this paper, it is worth summarising them in a single assumption. (Recall $\theta=(\mu_1-\mu_0)/\sigma$.)
\begin{assumption}\label{ass:mu}
We assume that $(\mu_0,\mu_1,\rho,\sigma)$ fulfil one of $(i)-(ii)$ above.
\end{assumption}

\begin{proposition}\label{prop:bf}
Under Assumption \ref{ass:mu}, for all $y\in \R$ we have $0\le b(y)<+\infty$ and moreover $b\in C(\R)$.
\end{proposition}
\begin{proof}
\emph{Step 1}. (\emph{Finiteness}.) 
Let us start by proving finiteness of the boundary with an argument by contradiction. Assume indeed that there is $y_0\in\R$ such that $[0,+\infty)\times\{y_0\}\in\cC$. Then by monotonicity of $b(\cdot)$ it must be that $[0,+\infty)\times[y_0,+\infty)\subset\cC$. Notice that we have already shown in step 2 of the proof of Proposition \ref{prop:b2} that this is impossible if $\mu_0+\mu_1<0$ and $\rho\ge \tfrac{\theta}{2\sigma}|\mu_0+\mu_1|$. Then it remains to prove the contradiction  for $\mu_0+\mu_1\ge 0$.

For future use, let us introduce
\begin{align}\label{XYcirc}
X^\circ_t=x+\mu_0t+\sigma\OW_t\quad\text{and}\quad Y^\circ_t=y-\tfrac{1}{2}(\mu_1+\mu_0)t.
\end{align}
Fix $t_0>0$ and define $y_1:=y_0+\tfrac{1}{2}(\mu_1+\mu_0)t_0$. Then by assumption it must be $\OP_{x,y_1}(\tau_*\ge t_0)=1$ for all $x\ge 0$. For $\tau_0:=\inf\{t\ge0\,:\,\XX_t=0\}$, using the strong Markov property and \eqref{uu} we obtain
\begin{align}\label{b1}
\uu(x,y_1)=\OE_{x,y_1}\left[e^{-\rho\tau_0}\uu(0,Y^{\circ}_{\tau_0})\mathds{1}_{\{\tau_0<\tau_*\}}-\rho\int_0^{\tau_0\wedge\tau_*}e^{-\rho t}g(X^\circ_t, Y^\circ_t)dt\right], 
\end{align}
where we use that for $t\le \tau_0$ we have $(\XX_t,\YY_t)=(X^\circ_t,Y^\circ_t)$, $\OP_{x,y_1}$-a.s.

From \eqref{sublin} we deduce that for some $c_{y_1}>0$, only depending on $y_1$, we have
\begin{align*}
e^{-\rho\tau_0}\uu(0,Y^{\circ}_{\tau_0})\le e^{-\rho\tau_0}\Big(1+ c_1 \,e^{\frac{\theta}{\sigma}Y^\circ_{\tau_0}}\Big)\le c_{y_1}\,e^{-\rho\tau_0},
\end{align*} 
where in the last inequality we have also used $\mu_0+\mu_1\ge 0$. Plugging the latter bound into \eqref{b1} and using that $\tau_*\ge t_0$ we get
\begin{align}\label{b2}
\uu(x,y_1)\le\,c_{y_1}\,\OE_{x,y_1}\left[e^{-\rho\tau_0}\right]-\rho\,\OE_{x,y_1}\left[\int_0^{\tau_0\wedge t_0}e^{-\rho t}g(X^\circ_t, Y^\circ_t)dt\right]. 
\end{align}
Taking $x\to\infty$ the first term on the right-hand side of \eqref{b2} goes to zero whereas the second one diverges to $+\infty$, because $\lim_{x\to\infty}\OP_{x,y_1}(\tau_0\ge t_0)= 1$ and $x\mapsto g(x,y)$ is increasing. Hence we have a contradiction.
\vspace{+4pt}

\emph{Step 2}. (\emph{Left-continuity}.) Using that $b(\cdot)$ is non-decreasing and that $\cS$ is closed we obtain that for any $y_0\in\R$ and any increasing sequence $y_n\uparrow y_0$ as $n\to\infty$ it must be that $\lim_{n\to\infty}(b(y_n),y_n)=(b(y_0-),y_0)\in\cS$, where $b(y_0-)$ is the left limit of $b$ at $y_0$. Then $b(y_0-)\ge b(y_0)$ by \eqref{cS3}, and since $b(y_n)\le b(y_0)$ for all $n\ge 1$ then $b$ must be left-continuous, hence lower semi-continuous. 
\vspace{+4pt}

\emph{Step 3}. (\emph{Right-continuity}.) The argument by contradiction that we are going to use draws from \cite{DeA15}. Assume there exists $y_0\in\R$ such that $b(y_0)<b(y_0+)$ and take $b(y_0)<x_1<x_2<b(y_0+)$ and a non-negative function $\phi\in C^\infty_c(x_1,x_2)$ such that $\int_{x_1}^{x_2}\phi(x)dx=1$. Thanks to Lemma \ref{lem:BVP1} (cf.~also \eqref{Lu}) we have 
\begin{align}\label{Luu}
(\cL_{X,Y}-\rho)\uu(x,y)=g(x,y)\quad\text{for $(x,y)\in(x_1,x_2)\times(y_0,+\infty)$.}
\end{align}

(\emph{Case $\mu_0+\mu_1\ge 0$}.) Let us first consider the case of $\mu_0+\mu_1\ge 0$. Recall that $\uu_y\ge 0$ in $\cC$ by \eqref{uhaty}.
Then, multiplying \eqref{Luu} by $\phi(\cdot)$ and integrating by parts we obtain
\begin{align*}
0\ge&\,-\tfrac{1}{2}(\mu_0+\mu_1)\int^{x_2}_{x_1}\uu_y(x,y)\phi(x)dx\\
=& \int^{x_2}_{x_1}\big[\rho\, g+\rho\, \uu-\mu_0\,\uu_x-\tfrac{\sigma^2}{2}\,\uu_{xx}\big](x,y)\phi(x)dx\\
=& \int^{x_2}_{x_1}\big[(\rho\, g+\rho\, \uu)(x,y)\phi(x)+\mu_0\,\uu(x,y)\phi'(x)-\tfrac{\sigma^2}{2}\,\uu(x,y)\phi''(x)\big]dx.
\end{align*}
Taking limits as $y\downarrow y_0$ and using dominated convergence and $\uu(x,y_0)=0$ we obtain a contradiction, that is 
\begin{align*}
0\ge&\,\rho\int^{x_2}_{x_1}g(x,y_0)\phi(x)dx>0.
\end{align*}
Hence $b(y_0)=b(y_0+)$.

(\emph{Case $\mu_0+\mu_1< 0$}.) Next consider the case $\mu_0+\mu_1<0$ and $\rho\ge \tfrac{\theta}{2\sigma}|\mu_0+\mu_1|$. Thanks to classical results on internal regularity of PDEs (e.g., \cite[Ch.~3, Thm.~10]{Fri}) we can differentiate \eqref{Luu} with respect to $x$ and say that $\uu_x\in C^{1,2}$ in $\cC$ and it solves
\begin{align}\label{pde-ux}
(\cL_{X,Y}-\rho)\uu_x=g_x\quad\text{in $(x_1,x_2)\times(y_0,+\infty)$}.
\end{align}
It is crucial now to recall that $\uu_x\le 0$ as it was shown in the proof of Proposition \ref{prop:b2}. 

For $y>y_0$, from \eqref{pde-ux} we get 
\begin{align}\label{pde-ux2}
&\int^{x_2}_{x_1}(\cL_{X,Y}\uu_x-\rho\,\uu_x-\rho\,g_x)(x,y)\phi(x)dx=0.
\end{align}
Defining $F_\phi(y):=\int^{x_2}_{x_1}\uu_{xy}(x,y)\phi(x)dx$ and using integration by parts, \eqref{pde-ux2} may be rewritten as
\begin{align*}
&\tfrac{1}{2}|\mu_0+\mu_1|F_\phi(y)\\
&=\int^{x_2}_{x_1}\big[\tfrac{\sigma^2}{2}\uu(x,y)\phi'''(x)-\mu_0\,\uu(x,y)\phi''(x)-\rho \uu(x,y)\phi'(x)+\rho g_x(x,y)\phi(x)\big]dx.
\end{align*}
Taking limits as $y\downarrow y_0$ and using $\uu(x,y_0)=0$ gives
\begin{align*}
F_\phi(y_0+)=\frac{2\rho}{|\mu_0+\mu_1|}\int^{x_2}_{x_1} g_x(x,y_0)\phi(x)dx\ge \rho_0>0
\end{align*}
for some $\rho_0$. Hence, there is $\eps>0$ such that $F_\phi(y)\ge \rho_0/2$ for $y\in(y_0,y_0+\eps)$. Then, from the definition of $F_\phi$, integration by parts and Fubini's theorem we find
\begin{align*}
\tfrac{1}{2}\rho_0\eps\le&\int_{y_0}^{y_0+\eps}F_\phi(y)dy=-\int_{x_1}^{x_2}\left(\int_{y_0}^{y_0+\eps}\uu_{y}(x,y)dy\right)\phi'(x)dx\\
=&-\int_{x_1}^{x_2}\uu(x,y_0+\eps)\phi'(x)dx=\int_{x_1}^{x_2}\uu_x(x,y_0+\eps)\phi(x)dx\le 0,
\end{align*}
where we also used $\uu(x,y_0)=0$. The contradiction above implies $b(y_0)=b(y_0+)$.
\end{proof}

Monotonicity of $b$ is the key to the regularity of the value function in this context. In fact we will use it to show that the first hitting time to $\cS$ coincides with the first hitting time to the interior of $\cS$. The latter, along with regularity (in the sense of diffusions) of $\partial\cS$, will be sufficient to prove that $\UU\in C^1((0,+\infty)\times\R)$, or equivalently $\OU\in C^1((0,+\infty)^2)$.

Let us introduce the first hitting times to $\cS$ and to $\cS^\circ:=\text{int}\{\cS\}$ as 
\begin{align}
\label{si-S}\sigma_*:=\inf\{t>0\,:\,(\XX_t,\YY_t)\in\cS\}\quad\text{and}\quad \sigma^\circ_*:=\inf\{t>0\,:\,(\XX_t,\YY_t)\in\cS^\circ\}.
\end{align}
Notice that continuity of paths for $(\XX,\YY)$ implies that $\tau_*=\sigma_*$ for all $(x,y)\in([0,+\infty)\times\R)\setminus\partial\cC$. It will be crucial to prove that the equality also holds at points of the boundary $(x,y)\in\partial\cC$. For future reference we define
\begin{align}\label{y^*_0}
y^*_0:=\inf\{y\in\R\,:\,(0,y)\in \cC\}\quad\text{(with $\inf\varnothing=+\infty$)}. 
\end{align}
\begin{lemma}\label{lem:ht}
Under Assumption \ref{ass:mu}, for all $(x,y)\in[0,+\infty)\times\R\setminus(0,y^*_0)$ we have
\begin{align}\label{Pss}
\OP_{x,y}(\sigma_*=\sigma_*^\circ)=1.
\end{align}
\end{lemma}
\begin{proof}
The statement is trivial for $(x,y)\in\cS^\circ$ and for $(x,y)\in\{0\}\times(-\infty,y^*_0)$ thanks to continuity of paths. It remains to consider $(x,y)\in\overline{\cC}\setminus(0,y^*_0)$, where $\overline{\cC}$ is the closure of $\cC$. First, we notice that thanks to monotonicity of $b$ we have
\begin{align}\label{Pss1}
\OP_{x,y}(\XX_{\sigma_*}>0)=1\quad \text{for all $(x,y)\in\overline{\cC}\setminus(0,y^*_0)$}.
\end{align}
If $\mu_1+\mu_2\le 0$, \eqref{Pss1} is obvious because $\YY$ is non-decreasing. If $\mu_1+\mu_2>0$, \eqref{Pss1} holds because 
\[
\OP_{x,y}(\XX_{\sigma_*}=0)=\OP_{x,y}((\XX_{\sigma_*},\YY_{\sigma_*})=(0,y^*_0))=0\quad\text{for all $(x,y)\in\overline{\cC}\setminus(0,y^*_0)$}.
\]
Let us now prove \eqref{Pss} in $\overline{\cC}\setminus(0,y^*_0)$.

In the case $\mu_1+\mu_0> 0$ the process $\YY$ has a negative drift and moves to the left at a constant rate, during excursions of $\XX$ away from $x=0$. Since $b$ is non-decreasing, then $t\mapsto b(\YY_t)$ is decreasing during excursions of $\XX$ away from $x=0$. It then becomes straightforward to verify \eqref{Pss}, due to the law of iterated logarithm for Brownian motion and \eqref{Pss1}.  

If $\mu_1+\mu_0=0$ the process $\YY$ only increases at times $t$ such that $\XX_t=0$, otherwise it stays constant. Then \eqref{Pss} holds, due to \eqref{Pss1} and because $\XX$ immediately enters intervals of the form $(x',+\infty)$, after reaching $x'$ (i.e., $x'$ is regular for $(x',+\infty)$). 

If $\mu_1+\mu_0<0$ the process $\YY$ increases. Moreover, during excursions of $\XX$ away from $x=0$, the rate of increase is constant. Recalling \eqref{Pss1}, we can therefore use \cite[Cor.~8]{CP15} to conclude that \eqref{Pss} indeed holds (see also a self contained proof in a setting similar to ours, in Appendix B of \cite{DeAGV17}). 
\end{proof}

We say that a boundary point $(x,y)\in\partial\cC$ is regular for the stopping set, in the sense of diffusions if 
\begin{align}\label{reg0}
\OP_{x,y}(\sigma_*>0)=0
\end{align}
(see \cite{BG}; see also \cite{DeAP18} for a recent account on this topic). Notice that, from the {\em $0-1$ Law}, if \eqref{reg0} fails then $\OP_{x,y}(\sigma_*>0)=1$. 

In case $\mu_0+\mu_1 \ge 0$, during excursions of $\XX$ away from zero the process $\YY$ is non-increasing. So the couple $(\XX,\YY)$ moves towards the left of the $(x,y)$-plane during such excursions (or $\YY$ is just constant if $\mu_0+\mu_1=0$). If $(\XX_0,\YY_0)=(x_0,y_0)\in\partial\cC$ with $x_0>0$, recalling that $b(\cdot)$ is non-decreasing, the law of iterared logarithm implies that $\OP_{x_0,y_0}(\sigma_*>0)=0$. So we can claim

\begin{proposition}\label{prop:reg1}
Assume $\mu_0+\mu_1 \ge 0$. Then all points $(x,y)\in\partial\cC$ with $x>0$ are regular for the stopping set, i.e.~\eqref{reg0} holds.
\end{proposition}  
To treat the regularity of $\partial\cC$ in the remaining case of $\mu_0+\mu_1<0$ we need to take a longer route because $(\XX,\YY)$ is now moving towards the right of the $(x,y)$-plane and in principle, when started from $\partial\cC$, it may `escape' from the stopping set. We shall prove below that this {\em is not} the case. For that, we first need to show that the \emph{smooth fit} holds at the boundary. Notice that this is the classical concept of smooth fit, i.e.~continuity of $z\mapsto\UU_x(z,y)$. Smooth fit in this sense does not imply that $(x,y)\mapsto\UU_x(x,y)$ is continuous across the boundary, which instead we will prove in Proposition \ref{prop:vC1}.
\begin{lemma}\label{lem:sm-f}
Assume $\mu_0+\mu_1 < 0$ and $\rho\ge \tfrac{\theta}{2\sigma}|\mu_0+\mu_1|$. For each $y\in\R$ we have $\UU_x(\cdot,y)\in C(0,+\infty)$ (equiv.~$\uu_x(\cdot,y)\in C(0,+\infty)$).
\end{lemma}
\begin{proof}
From 
\begin{align*}
\tfrac{\sigma^2}{2}\uu_{xx}(x,y)=&\,\rho g(x,y)+\rho\uu(x,y)-\mu_0\uu_x(x,y)+\tfrac{1}{2}(\mu_0+\mu_1)\uu_y(x,y),\quad\text{for $(x,y)\in\cC$}
\end{align*}
and using \eqref{Lip-OU} (which clearly implies Lipschitz continuity of $\UU$ as well) we see that for any bounded set $B$ it must be that 
\begin{align}\label{IT}
\text{$\uu_{xx}$ is bounded on the closure of $B\cap\cC$.} 
\end{align}
This fact will be used later to justify the use of It\^o-Tanaka formula in \eqref{sm-f1}.

We establish the smooth fit with an argument by contradiction. The first step is to recall that $\uu_x\le 0$ in $\cC$ as it was verified in the proof of Proposition \ref{prop:b2}. Second, notice that any $(x_0,y_0)\in\partial\cC$ must be of the form $(b(y_0),y_0)$ due to continuity of $y\mapsto b(y)$ (Proposition \ref{prop:bf}). Next, assume that for some $y_0$ and $x_0=b(y_0)>0$ we have $\uu_x(x_0-,y_0)<-\delta_0$ for some $\delta_0>0$, where $\uu_x(x_0-,y_0)$ exists due to \eqref{IT}. Take a bounded rectangular neighbourhood $B$ of $(x_0,y_0)$ such that $B\cap(\{0\}\times\R)=\varnothing$ and let $\tau_B:=\inf\{t\ge0\,:\,(\XX_t,\YY_t)\notin B\}$. Then from the super-martingale property of $\UU$ \eqref{supermg}, using that $A_{\tau_B\wedge t}=0$ for all $t\ge0$ and recalling \eqref{XYcirc}, we have
\begin{align*}
\uu(x_0,y_0)\ge \OE_{x_0,y_0}\left[e^{-\rho(\tau_B\wedge t)}\uu(X^\circ_{\tau_B\wedge t},Y^\circ_{\tau_B\wedge t})-\rho\int_0^{\tau_B\wedge t}e^{-\rho s}g(X^\circ_s,Y^\circ_s)ds\right].
\end{align*}
Now we notice that $t\mapsto Y^\circ_{\tau_B\wedge t}$ is increasing. Moreover, recalling \eqref{uhaty}, we have $\uu_y\ge 0$ in $\cC$. This implies $\uu(X^\circ_{\tau_B\wedge t},Y^\circ_{\tau_B\wedge t})\ge \uu(X^\circ_{\tau_B\wedge t},y_0)$, $\OP_{x_0,y_0}$-a.s. Finally, observing that $g$ is bounded on $B$ we obtain
\begin{align}\label{sm-f1}
\uu(x_0,y_0)\ge \OE_{x_0,y_0}\left[e^{-\rho(\tau_B\wedge t)}\uu(X^\circ_{\tau_B\wedge t},y_0)-c (\tau_B\wedge t)\right]
\end{align}
for some $c=c(B)>0$ that depends on the set $B$ and {\em will vary from line to line below}. 

As anticipated, we can now use It\^o-Tanaka formula in \eqref{sm-f1} thanks to \eqref{IT}. We let $\cL_X=\tfrac{\sigma^2}{2}\partial_{xx}+\mu_0\partial_x$, denote the local time of $X^\circ$ at $x_0$ by $L^{x_0}$, and notice also that $\uu_{xx}(\,\cdot\,,y_0)=0$ for $x>x_0$. Then
\begin{align}\label{sm-f2}
0\ge&\, \OE_{x_0,y_0}\left[\int_0^{\tau_B\wedge t}e^{-\rho s}(\cL_X-\rho)\uu(X^\circ_{s},y_0)\mathds{1}_{\{X^\circ_s\neq x_0\}}ds-c (\tau_B\wedge t)\right]\\[+4pt]
&-\OE_{x_0,y_0}\left[\int_0^{\tau_B\wedge t}e^{-\rho s}\uu_x(x_0-,y_0)dL^{x_0}_s\right]\nonumber\\[+4pt]
\ge &\,\delta_0 e^{-\rho t}\OE_{x_0,y_0}\big[L^{x_0}_{\tau_B\wedge t}\big]-c\,\OE_{x_0,y_0}\big[\tau_B\wedge t\big],\nonumber
\end{align}
where in the final inequality we used that $(\cL_X-\rho)\uu$ is bounded on $B$.
Letting $t\to0$ the inequality in \eqref{sm-f2} leads to a contradiction because $\OE_{x_0,y_0}\big[L^{x_0}_{\tau_B\wedge t}\big]\approx\sqrt{t}$ whereas $\OE_{x_0,y_0}\big[\tau_B\wedge t\big]\approx t$ (the argument is similar to the one used to prove Proposition \ref{prop:notempty}. See also, e.g., \cite[Lem.~6.5]{DeAGV17} or \cite[Lem.~13]{Pe17}).

Hence the claim is proved.
\end{proof}

Next we establish regularity of $\partial\cC$ in the sense of diffusions, when $\mu_0+\mu_1<0$.
\begin{proposition}\label{lem:reg}
Assume $\mu_0+\mu_1<0$ and $\rho\ge \tfrac{\theta}{2\sigma}|\mu_0+\mu_1|$. Then all points $(x,y)\in\partial\cC$ with $x>0$ are regular for the stopping set, i.e.~\eqref{reg0} holds.
\end{proposition}
\begin{proof}
The idea of proof is to show that if $\OP_{x_0,y_0}(\sigma_*>0)=1$ for some $(x_0,y_0)\in\partial\cC$ then $\uu_x(x_0-,y_0)<0$, which contradicts Lemma \ref{lem:sm-f}.
\vspace{+4pt}

\emph{Step 1.} [\emph{Upper bound on $\uu_x$}.] Let us start by fixing $(x,y)\in\cC$.
It is convenient to rewrite $\uu$ in the following form: let $\tau_\eps:=\inf\{t\ge0\,:\,\XX_t=\eps\}$, for $\eps\ge 0$, then by strong Markov property we have
\begin{align}\label{uu2}
\uu(x,y)=&\,\sup_{\tau}\OE_{x,y}\Big[\mathds{1}_{\{\tau>\tau_\eps\}}e^{-\rho\tau_\eps}\uu(\eps,Y^\circ_{\tau_\eps})-\rho\int_{0}^{\tau_\eps\wedge\tau} e^{-\rho t}g(X^\circ_t,Y^\circ_t)dt\Big],
\end{align}
where we used that $(\XX_t,\YY_t)=(X^\circ_t,Y^\circ_t)$ for $t\le\tau_\eps$, $\OP_{x,y}$-a.s., with the notation of \eqref{XYcirc}.

Notice that $\tau_\eps$ is independent of $y$ and therefore we can say $\tau_\eps=\tau_\eps(x)$. Moreover, due to \eqref{XX2}, it is clear that 
\begin{align}\label{taueps}
\tau_0(x-\eps)=\tau_\eps(x),\quad\OP-a.s.
\end{align}
Now, fix $\eps>0$, denote $X^{\circ,\eps}_t=x-\eps+\mu_0t+\sigma\OW_t$ and $\tau^\eps_0=\tau_0(x-\eps)$,  $\OP$-a.s., and take $\tau'=\tau_*(x,y)$, which is sub-optimal for $\uu(x-\eps,y)$.
Then we obtain 
\begin{align}\label{reg1}
\uu(x-\eps,y)\ge \OE\Big[\mathds{1}_{\{\tau'>\tau_0^\eps\}}e^{-\rho\tau^\eps_0}\uu(0,Y^\circ_{\tau^\eps_0})-\rho\int_{0}^{\tau^\eps_0\wedge\tau'}\! e^{-\rho t}g(X^{\circ,\eps}_t,Y^\circ_t)dt\Big].
\end{align}

Thanks to \eqref{taueps}, in \eqref{reg1} we can replace $\tau^\eps_0$ with $\tau_\eps$ as in \eqref{uu2}. Then,
subtracting \eqref{reg1} from \eqref{uu2} we obtain
\begin{align}
\uu&\,(x,y)-\uu(x-\eps,y)\\[+4pt]
\le &\,\OE\Big[\mathds{1}_{\{\tau'>\tau_\eps\}}e^{-\rho\tau_\eps}\left(\uu(\eps,Y^\circ_{\tau_\eps})-\uu(0,Y^\circ_{\tau_\eps})\right)\Big]\nonumber\\[+4pt]
&\,+\OE\left[\rho\int_{0}^{\tau_\eps\wedge\tau'} e^{-\rho t}\left(g(X^{\circ,\eps}_t,Y^\circ_t)-g(X^\circ_t,Y^\circ_t)\right)dt\right]\nonumber\\[+4pt]
\le &\, \OE\left[\rho\int_{0}^{\tau_\eps\wedge\tau'} e^{-\rho t}\left(g(X^{\circ,\eps}_t,Y^\circ_t)-g(X^\circ_t,Y^\circ_t)\right)dt\right],\nonumber
\end{align}
where the last inequality uses $\uu_x\le 0$ in $\cC$ (cf.~proof of Proposition \ref{prop:b2}) and $(\eps,Y^\circ_{\tau_\eps})\in\cC$ on $\{\tau'>\tau_\eps\}$. Now we can divide by $\eps$ and let $\eps\to0$. Using that $\tau_\eps\downarrow \tau_0$ and recalling $\tau'=\tau_*(x,y)$ we obtain
\begin{align}\label{ub}
\uu_x(x,y)\le -\rho \OE_{x,y}\left[\int_0^{\tau_*\wedge\tau_0}e^{-\rho t}g_x(X^\circ_t,Y^\circ_t)dt\right].
\end{align}
\vspace{+4pt}

\emph{Step 2.} [\emph{Non-smooth fit}.] Assume $(x_0,y_0)\in\partial\cC$ and $\OP_{x_0,y_0}(\sigma_*>0)=1$. Take an increasing sequence $x_n\uparrow x_0$ and denote $\tau_*^n=\tau_*(x_n,y_0)$. Notice that $\tau_*^n=\sigma_*^n=\sigma_*(x_n,y_0)$ for all $n\ge 1$, due to continuity of paths. Moreover, $\sigma_*^n$ is decreasing in $n$, with $\sigma_*^n\ge \sigma_*=\sigma_*(x_0,y_0)$, because $x\mapsto X^{x}_t$ is increasing and $\cS$ is of the form \eqref{cS3}. Setting $\tau_0^n=\inf\{t\ge0\,:\,X^{\circ,n}_t=0\}$ it is also easy to check that $\tau^n_0\uparrow\tau_0$ as $n\to\infty$. Then, denoting $\sigma^\infty:=\lim_{n\to\infty}\sigma^n_*$, we have
\[
\sigma^\infty\wedge\tau_0=\lim_{n\to\infty}\sigma^n_*\wedge \tau_0^n\ge \sigma_*\wedge\tau_0,\qquad\OP-a.s.
\]

Since $g_x\ge 0$ we can use monotone convergence and \eqref{ub} to get
\begin{align}\label{contra}
u_x(x_0-,y_0)=\lim_{n\to\infty}\uu_x(x_n,y_0)\le -\rho\, \OE_{x,y}\left[\int_0^{\sigma^\infty\wedge\tau_0}e^{-\rho t}g_x(X^\circ_t,Y^\circ_t)dt\right]<0,
\end{align}
where the final inequality holds because $\OP_{x_0,y_0}(\sigma^\infty\ge \sigma_*>0)=1$ by assumption. 
\vspace{+4pt}

The result in \eqref{contra} contradicts Lemma \ref{lem:sm-f}, hence $\OP_{x_0,y_0}(\sigma_*>0)=0$.
\end{proof}

As corollary to Lemma \ref{lem:ht} and Propositions \ref{prop:reg1} and \ref{lem:reg}, we have
\begin{corollary}\label{cor:ht}
Under Assumption \ref{ass:mu}, for all $(x,y)\in[0,+\infty)\times\R\setminus (0,y^*_0)$ we have
\begin{align}
\OP_{x,y}(\tau_*=\sigma_*=\sigma_*^\circ)=1.
\end{align}
\end{corollary}
This corollary is important to determine continuity of the stopping times with respect to the initial position of the process, at all points of the state space.
\begin{proposition}\label{prop:c-ht}
Let Assumption \ref{ass:mu} hold. Then for any $(x,y)\in[0,+\infty)\times\R\setminus(0,y^*_0)$ and any sequence $(x_n,y_n)\to(x,y)$ it holds 
\begin{align}\label{c-ht}
\lim_{n\to\infty}\tau_*(x_n,y_n)=\tau_*(x,y),\quad\OP-a.s.
\end{align}
In particular, for $(x,y)\in\partial\cC\setminus(0,y^*_0)$ the limit is zero.
\end{proposition}
\begin{proof}
Let us fix $(x,y)\in[0,+\infty)\times\R$. For simplicity, in the rest of this proof all stopping times depending on $(x_n,y_n)$ will be denoted by $\tau_n$, $\sigma_n$ or $\sigma^\circ_n$ whereas those depending on $(x,y)$ will be denoted by $\tau$, $\sigma$ or $\sigma^\circ$, as appropriate. 
\vspace{+4pt}

\emph{Step 1.} [\emph{Lower semi-continuity}.] Here we show that 
\begin{align}\label{cht0}
\liminf_{n\to \infty}\tau_n\ge \tau,\quad\OP-a.s.
\end{align}
Fix $\omega\in\Omega$ outside of a null set. Then, if $\tau(\omega)=0$ the result is trivial. Assume $\tau(\omega)>\delta$ for some $\delta>0$. Then, recalling that the boundary is continuous (Proposition \ref{prop:bf}) there exists $c_{\delta,\omega}>0$ such that
\begin{align}
b(\YY^{x,y}_t(\omega))-\XX^x_t(\omega)> c_{\delta,\omega},\qquad\text{for all $t\in[0,\delta]$.}
\end{align}
Notice that the map $(t,x',y')\mapsto b(\YY^{x',y'}_t(\omega))-\XX^{x'}_t(\omega)$ is uniformly continuous on any compact $[0,\delta]\times K$. Then we can find $\overline n_\omega\ge 1$ sufficiently large that, for all $n\ge \overline n_\omega$ 
\begin{align}\label{cht1}
b(\YY^{x_n,y_n}_t(\omega))-\XX^{x_n}_t(\omega)> c_{\delta,\omega},\qquad\text{for all $t\in[0,\delta]$}.
\end{align}
Hence \eqref{cht1} implies $\liminf_{n\to\infty}\tau_n(\omega)\ge \delta$. Since $\omega,\delta$ were arbitrary, we obtain \eqref{cht0}.

\emph{Step 2.} [\emph{Upper semi-continuity}.] Here we show that 
\begin{align}\label{cht2}
\limsup_{n\to \infty}\sigma^\circ_n\le \sigma^\circ,\quad\OP-a.s.
\end{align}
Fix $\omega\in\Omega$ outside of a null set. Then, if $\sigma^\circ(\omega)=+\infty$ the result is trivial. Assume $\sigma^\circ(\omega)<\delta$ for some $\delta>0$. Then, recalling that the boundary is continuous (Proposition \ref{prop:bf}) there exists $t\le\delta$ such that
\begin{align*}
b(\YY^{x,y}_t(\omega))<\XX^x_t(\omega).
\end{align*}
By continuity of $(x',y')\mapsto b(\YY^{x',y'}_t(\omega))-\XX^{x'}_t(\omega)$ we can find  $\overline n_\omega\ge 1$ sufficiently large that, for all $n\ge \overline n_\omega$
\begin{align*}
b(\YY^{x_n,y_n}_t(\omega))<\XX^{x_n}_t(\omega).
\end{align*}
Hence $\limsup_{n\to\infty}\sigma^\circ_n\le \delta$. Since $\omega,\delta$ were arbitrary \eqref{cht2} follows.
\vspace{+4pt}

Combining step 1 and 2 above with Corollary \ref{cor:ht} we obtain \eqref{c-ht}.
\end{proof}

In order to finally prove that $\UU\in C^1((0,+\infty)\times\R)$ we would like to have a fully probabilistic representation of $\nabla_{x,y}\UU$. If obtaining $\UU_y$ in \eqref{UUy} was relatively easy, we now need more care for $\UU_x$. First of all, recalling the explicit dynamics of $(\XX,\YY,A)$ from \eqref{A}, \eqref{XX2} and \eqref{YY2}, and denoting by $\partial_x^+$ and $\partial_x^-$ the right and left partial derivatives with respect to $x$, we observe that for all $(x,y)\in(0,+\infty)\times\R$ and $t\ge 0$
\begin{align}
\label{dxX}&\partial^+_x\XX^x_t=\mathds{1}_{\{x\ge S^{\mu_0,\sigma}_t\}}\,, \qquad \partial^-_x\XX^x_t=\mathds{1}_{\{x> S^{\mu_0,\sigma}_t\}}\,,\\[+4pt]
\label{dxY}&\partial^+_x\YY^{x,y}_t=\partial^+_xA^{x}_t=-\mathds{1}_{\{x< S^{\mu_0,\sigma}_t\}}\,, \quad \partial^-_x\YY^{x,y}_t=\partial^-_xA^{x}_t=-\mathds{1}_{\{x\le S^{\mu_0,\sigma}_t\}},
\end{align}
where we also recall the notation $S^{\mu_0,\sigma}_t=\sup_{0\le s\le t}(-\mu_0 s-\sigma \OW_s)$.

Recalling $y^*_0$ from \eqref{y^*_0} and $\tau_0=\inf\{t\ge0\,:\,\XX_t=0\}$, the same rationale as in \eqref{Pss1} and Corollary \ref{cor:ht} give
\begin{align}\label{pt0t*}
\OP_{x,y}(S^{\mu_0,\sigma}_{\tau_*}=x)=\OP_{x,y}(\tau_*=\tau_0)=\OP_{x,y}((\XX_{\tau_*},\YY_{\tau_*})=(0,y^*_0))=0 
\end{align}
for any $(x,y)\in([0,+\infty)\times\R)\setminus(0,y^*_0)$. Then for all $(x,y)\in(0,+\infty)\times\R$ and $\OP$-a.s.~we have 
\begin{align}
\label{dxX1}&\partial^+_x\XX^x_{\tau_*}=\partial^-_x\XX^x_{\tau_*}=\mathds{1}_{\{x\ge S^{\mu_0,\sigma}_{\tau_*}\}}\,,\\[+4pt]
\label{dxY1}&\partial^+_x\YY^{x,y}_{\tau_*}=\partial^-_x\YY^{x,y}_{\tau_*}=\partial^+_xA^{x}_{\tau_*}=\partial^-_xA^{x}_{\tau_*}=-\mathds{1}_{\{x\le S^{\mu_0,\sigma}_{\tau_*}\}}.
\end{align}

Let us now obtain the probabilistic representation of $\UU_x$.
\begin{lemma}\label{lem:UUx}
For all $(x,y)\in((0,+\infty)\times\R)\setminus\partial\cC$ we have
\begin{align}\label{UUx}
\UU_x&(x,y)\nonumber\\[+4pt]
=&-\tfrac{2\mu_0}{\sigma^2}\OE_{x,y}\left[\mathds{1}_{\{\tau_*<+\infty\}}\mathds{1}_{\{x\le S^{\mu_0,\sigma}_{\tau_*}\}}e^{\frac{2\mu_0}{\sigma^2}A_{\tau_*}-\rho\tau_*}\left(1+e^{\frac{\theta}{\sigma}(\XX_{\tau_*}+\YY_{\tau_*})}\right)\right]\\[+4pt]
&+\tfrac{\theta}{\sigma}\OE_{x,y}\left[\mathds{1}_{\{\tau_*<+\infty\}}\left(\mathds{1}_{\{x\ge S^{\mu_0,\sigma}_{\tau_*}\}}-\mathds{1}_{\{x\le S^{\mu_0,\sigma}_{\tau_*}\}}\right)e^{\frac{2\mu_0}{\sigma^2}A_{\tau_*}-\rho\tau_*}e^{\frac{\theta}{\sigma}(\XX_{\tau_*}+\YY_{\tau_*})}\right].\nonumber
\end{align}
\end{lemma}
\begin{proof}
The result is trivial for $(x,y)\in\cS^\circ$, because $\tau_*=0$. For $(x,y)\in\cC$, we recall that $\UU_x$ is well defined (Lemma \ref{lem:BVP1}), we take $\eps>0$ and denote by $\tau=\tau_*(x,y)$ the optimal stopping time for $\UU(x,y)$. For any $t>0$, using the (super)-martingale property \eqref{supermg}--\eqref{mg} we have that 
\begin{align}
\UU&(x+\eps,y)-\UU(x,y)\\[+4pt]
\ge&\,\OE\left[e^{\frac{2\mu_0}{\sigma^2}A^{x+\eps}_{\tau\wedge t}-\rho(\tau\wedge t)}\UU(\XX^{x+\eps}_{\tau\wedge t},\YY^{x+\eps,y}_{\tau\wedge t})-e^{\frac{2\mu_0}{\sigma^2}A^{x}_{\tau\wedge t}-\rho(\tau\wedge t)}\UU(\XX^{x}_{\tau\wedge t},\YY^{x,y}_{\tau\wedge t})\right]\nonumber\\[+4pt]
\ge&\,\OE\left[\mathds{1}_{\{\tau\le t\}}\left(e^{\frac{2\mu_0}{\sigma^2}A^{x+\eps}_{\tau}-\rho\tau}g(\XX^{x+\eps}_{\tau},\YY^{x+\eps,y}_{\tau})-e^{\frac{2\mu_0}{\sigma^2}A^{x}_{\tau}-\rho\tau}g(\XX^{x}_{\tau},\YY^{x,y}_{\tau})\right)\right]\nonumber\\[+4pt]
&+\OE\left[\mathds{1}_{\{\tau> t\}}\left(e^{\frac{2\mu_0}{\sigma^2}A^{x+\eps}_{t}-\rho t}\UU(\XX^{x+\eps}_{t},\YY^{x+\eps,y}_{t})-e^{\frac{2\mu_0}{\sigma^2}A^{x}_{t}-\rho t}\UU(\XX^{x}_{t},\YY^{x,y}_{t})\right)\right].\nonumber
\end{align}
Dividing the above expressions by $\eps$, letting $\eps\to0$ and using \eqref{Lip-OU} and the right derivatives in \eqref{dxX}, \eqref{dxY}, \eqref{dxX1} and \eqref{dxY1}, we find a lower bound for $\UU_x$, that is
\begin{align}\label{UUx1}
\UU_x(x,y)\ge&-\tfrac{2\mu_0}{\sigma^2}\OE\left[\mathds{1}_{\{\tau\le t\}}\mathds{1}_{\{x\le S^{\mu_0,\sigma}_{\tau}\}}e^{\frac{2\mu_0}{\sigma^2}A^x_{\tau}-\rho\tau}\left(1+e^{\frac{\theta}{\sigma}(\XX^x_{\tau}+\YY^{x,y}_{\tau})}\right)\right]\\[+4pt]
&+\tfrac{\theta}{\sigma}\OE\left[\mathds{1}_{\{\tau\le t\}}\left(\mathds{1}_{\{x\ge S^{\mu_0,\sigma}_{\tau}\}}-\mathds{1}_{\{x\le S^{\mu_0,\sigma}_{\tau}\}}\right)e^{\frac{2\mu_0}{\sigma^2}A^x_{\tau}-\rho\tau}e^{\frac{\theta}{\sigma}(\XX^x_{\tau}+\YY^{x,y}_{\tau})}\right]\nonumber\\[+4pt]
&+r(t,x,y),\nonumber
\end{align} 
where (notice that $c>0$ below is the same as in \eqref{Lip-OU})
\begin{align}
r(t,x,y)=&-\big|\tfrac{2\mu_0}{\sigma^2}\big|\OE\left[\mathds{1}_{\{\tau>t\}}\mathds{1}_{\{S^{\mu_0,\sigma}_t>x\}}e^{\frac{2\mu_0}{\sigma^2}A^x_t-\rho t}\UU(\XX^x_t,\YY^{x,y}_t)\right]\\[+4pt]
&+c\,\OE\left[\mathds{1}_{\{\tau>t\}}e^{\frac{2\mu_0}{\sigma^2}A^x_t-\rho t}g(\XX^x_t,\YY^{x,y}_t)\right]\nonumber\\[+4pt]
&+\tfrac{\theta}{\sigma}\OE\left[\mathds{1}_{\{\tau> t\}}\left(\mathds{1}_{\{x\ge S^{\mu_0,\sigma}_{t}\}}-\mathds{1}_{\{x\le S^{\mu_0,\sigma}_{t}\}}\right)e^{\frac{2\mu_0}{\sigma^2}A^x_{t}-\rho t}e^{\frac{\theta}{\sigma}(\XX^x_{t}+\YY^{x,y}_{t})}\right].\nonumber
\end{align}
Using \eqref{tran} for the three terms above, it is not hard to verify that $\lim_{t\to\infty}r(t,x,y)=0$ (notice that $\UU(x,y)\ge g(x,y)\ge e^{\frac{\theta}{\sigma}(x+y)}\ge 0$). Hence, taking limits as $t\to\infty$ in \eqref{UUx1} and recalling also \eqref{simple}, dominated convergence gives
\begin{align*}
\UU_x(x,y)\ge&-\tfrac{2\mu_0}{\sigma^2}\OE\left[\mathds{1}_{\{\tau<\infty\}}\mathds{1}_{\{x\le S^{\mu_0,\sigma}_{\tau}\}}e^{\frac{2\mu_0}{\sigma^2}A^x_{\tau}-\rho\tau}\left(1+e^{\frac{\theta}{\sigma}(\XX^x_{\tau}+\YY^{x,y}_{\tau})}\right)\right]\\[+4pt]
&+\tfrac{\theta}{\sigma}\OE\left[\mathds{1}_{\{\tau<\infty\}}\left(\mathds{1}_{\{x\ge S^{\mu_0,\sigma}_{\tau}\}}-\mathds{1}_{\{x\le S^{\mu_0,\sigma}_{\tau}\}}\right)e^{\frac{2\mu_0}{\sigma^2}A^x_{\tau}-\rho\tau}e^{\frac{\theta}{\sigma}(\XX^x_{\tau}+\YY^{x,y}_{\tau})}\right].\nonumber
\end{align*} 

In order to obtain an upper bound for $\UU_x$ we can employ symmetric arguments, using again $\tau=\tau_*(x,y)$, to estimate $\eps^{-1}(\UU(x,y)-\UU(x-\eps,y))$. It is not hard to check that the upper bound is the same as the lower bound, hence \eqref{UUx} holds.
\end{proof}

Thanks to continuity of the optimal stopping times and to the probabilistic representations of $\UU_x$ and $\UU_y$ we can state our next result (see also \cite{DeAP18} for general results in this direction).
\begin{proposition}\label{prop:vC1}
Under Assumption \ref{ass:mu} we have $\UU\in C^1((0,\infty)\times\R)$. 
\end{proposition}
\begin{proof}
Trivially $\UU\in C^1$ in $\cS^\circ$ and moreover $\UU\in C^1$ in $\cC\setminus(\{0\}\times\R)$, due to Lemma \ref{lem:BVP1}. It only remains to prove that $\nabla_{x,y}\UU$ is continuous across the boundary $\partial\cC$. Let us consider the case of $\UU_x$, as the proof for $\UU_y$ follows the same arguments.

Let $(x_0,y_0)\in\partial\cC$, with $x_0>0$, and let $(x_n,y_n)_{n\ge 1}$ be a sequence in $\cC$ converging to $(x_0,y_0)$ as $n\to\infty$. Thanks to Proposition \ref{prop:c-ht} we have that $\tau_*(x_n,y_n)\to \tau_*(x_0,y_0)=0$, $\OP$-a.s.~as $n\to\infty$. To simplify notation let $\tau_n:=\tau_*(x_n,y_n)$. 

Let $t>0$ be given and notice that on $\{\tau_n>t\}$ one has $(\XX_t,\YY_t)\in\cC$, $\OP_{x_n,y_n}$-a.s.~so that $\UU_x(\XX^{x_n}_t,\YY^{x_n,y_n}_t)$ may be represented by using \eqref{UUx}. Hence, tower property of conditional expectation and Markov property allow us to write \eqref{UUx} as 
\begin{align}\label{UUx2}
\UU_x&(x_n,y_n)\nonumber\\[+4pt]
=&-\tfrac{2\mu_0}{\sigma^2}\OE_{x,y}\left[\mathds{1}_{\{\tau_n\le t\}}\mathds{1}_{\{x_n\le S^{\mu_0,\sigma}_{\tau_n}\}}e^{\frac{2\mu_0}{\sigma^2}A_{\tau_n}-\rho\tau_n}\left(1+e^{\frac{\theta}{\sigma}(\XX_{\tau_n}+\YY_{\tau_n})}\right)\right]\\[+4pt]
&+\tfrac{\theta}{\sigma}\OE_{x,y}\left[\mathds{1}_{\{\tau_n\le t\}}\left(\mathds{1}_{\{x_n\ge S^{\mu_0,\sigma}_{\tau_n}\}}-\mathds{1}_{\{x_n\le S^{\mu_0,\sigma}_{\tau_n}\}}\right)e^{\frac{2\mu_0}{\sigma^2}A_{\tau_n}-\rho\tau_n}e^{\frac{\theta}{\sigma}(\XX_{\tau_n}+\YY_{\tau_n})}\right]\nonumber\\[+4pt]
&+\OE_{x_n,y_n}\left[\mathds{1}_{\{\tau_n>t\}}e^{\frac{2\mu_0}{\sigma^2}A_{t}-\rho t}\UU_x(\XX_t,\YY_t)\right].\nonumber
\end{align}

Now we want to take limits as $n\to\infty$ and use that $\tau_n\to 0$ in \eqref{UUx2} to show that $\UU_x(x_n,y_n)\to g_x(x_0,y_0)$. For that, first notice that $x\mapsto \mathds{1}_{\{x\le0\}}$ and $x\mapsto \mathds{1}_{\{x\ge0\}}$ are continuous on $(-\infty,0)$ and in particular at $-x_0$. Since we also have 
\[
\lim_{n\to\infty}S^{\mu_0,\sigma}_{\tau_n}-x_n\to-x_0<0,
\]
then $\OP$-a.s.~we have 
\[
\lim_{n\to\infty}\mathds{1}_{\{S^{\mu_0,\sigma}_{\tau_n}-x_n\ge 0\}}=0\quad\text{and}\quad\lim_{n\to\infty}\mathds{1}_{\{S^{\mu_0,\sigma}_{\tau_n}-x_n\le 0\}}= 1.
\]
Moreover, thanks to \eqref{Lip-OU} and \eqref{simple} we can invoke dominated convergence to take limits inside expectations in \eqref{UUx2}. This gives
\begin{align}
\lim_{n\to\infty}\UU_x(x_n,y_n)=\tfrac{\theta}{\sigma}\exp{\tfrac{\theta}{\sigma}(x_0+y_0)}=g_x(x_0,y_0)
\end{align}  
where we also used that $\lim_{n\to\infty}\mathds{1}_{\{\tau_n>t\}}=0$.

By arbitrariness of $(x_0,y_0)$ and of the sequence $(x_n,y_n)$ we conclude that $\UU_x$ is continuous across $\partial\cC\setminus(0,y^*_0)$. Similar arguments, applied to \eqref{UUy}, allow to show that $\UU_y$ is continuous across $\partial\cC\setminus(0,y^*_0)$ as well.
\end{proof}
Our proposition above has a simple corollary. Recall that $\overline{\cC}$ is the closure of $\cC$.
\begin{corollary}\label{cor:vC1}
Under Assumption \ref{ass:mu} we have $\OU\in C^1((0,+\infty)^2)$ and $U\in C^1((0,+\infty)\times(0,1))$. Moreover $\UU_{xx}$ is continuous on $\overline{\cC}\setminus(0,y^*_0)$ with 
\begin{align}\label{Uxx}
\UU_{xx}(x,y)=\tfrac{2\rho}{\sigma^2}g(x,y)+g_{xx}(x,y),\quad\text{for all $(x,y)\in\partial\cC\setminus(0,y^*_0)$}.
\end{align}
\end{corollary}
\begin{proof}
The first claim follows from Proposition \ref{prop:vC1}, \eqref{OU-UU} and \eqref{U-OU}. For the second claim we need  
\begin{align*}
\tfrac{\sigma^2}{2}\uu_{xx}+\mu_0\uu_x-\tfrac{1}{2}(\mu_0+\mu_1)\uu_y-\rho\uu=\rho g,\quad\text{in $\cC$},
\end{align*}
where $\uu=\UU-g$. Then \eqref{Uxx} follows by taking limits as $\cC\ni(x,y)\to (x_0,y_0)\in\partial\cC$, with $x_0>0$, and using $\uu_x=\uu_y=\uu=0$ on $\partial \cC\setminus(0,y^*_0)$.
\end{proof}
\begin{remark}\label{rem:reg}
Notice that due to internal regularity results for parabolic PDEs (cf.~\cite[Ch.~3, Thm.~10]{Fri}), and thanks to Lemma \ref{lem:BVP1}, we know that $\UU\in C^\infty$ in $\cC\setminus(\{0\}\times\R)$. This implies that also $\OU$ and $U$ belong to $C^\infty$ in $\cC\setminus(\{0\}\times\R)$. 
\end{remark}

\subsection{Reflection, creation and inverse of the boundary}\label{sec:refl-cr}

Recall that we conjectured that the boundary condition \eqref{HJBu3} holds for $U$ in \eqref{U}. We will now verify that this is indeed true, provided that we understand it in the limit as $x\downarrow 0$, for each given $\pi\in(0,1)$. Let us start by recalling that \eqref{U-OU} holds with $\varphi=\pi/(1-\pi)$.
Then, thanks to Remark \ref{rem:reg}, $U$ satisfies 
\begin{align}\label{HJBu3b}
\tfrac{\sigma^2}{2}U(0+,\pi)+&\,\sigma\theta\pi(1-\pi)U_\pi(0+,\pi)\\[+3pt]
+&\,(\mu_0+\hat\mu\pi)U(0+,\pi)=0\quad\text{for $\pi\in(0,1)$ s.t.~$(0,\pi)\in\cC$}\nonumber
\end{align}
if and only if 
\begin{align}\label{HJBu4}
\tfrac{\sigma^2}{2}\OU_x(0+,\varphi)\!+\!\hat \mu\varphi\OU(0+,\varphi)\!+\!\mu_0\OU(0+,\varphi)=0,\:\:\:\text{for all $\varphi\!>\!0$ s.t.~$(0,\varphi)\!\in\!\cC$}.
\end{align}
Further, recalling that $\UU(x,y)=\OU(x,\exp\tfrac{\theta}{\sigma}(x+y))$, then \eqref{HJBu4} holds if an only if
\begin{align}\label{HJBu5}
\tfrac{\sigma^2}{2}(\UU_x+\UU_y)(0+,y)+\mu_0\UU(0+,y)=0, \quad\text{for all $y\in\R$ s.t.~$(0,y)\in\cC$.}
\end{align}
The boundary condition \eqref{HJBu3b} is what we refer to as {\em reflection and creation} condition. Notice that $\{y\in\R\,:\,(0,y)\in\cC\}\neq\varnothing$ was proven in Proposition \ref{prop:notempty}. 

\begin{proposition}\label{prop:b-cond}
The boundary condition \eqref{HJBu3b} holds.
\end{proposition}
\begin{proof}
We will prove \eqref{HJBu5}. Fix $y\in\R$ with $(0,y)\in\cC$ and take a sequence $x_n\downarrow 0$ as $n\to\infty$. Notice that $\XX^{x_n}$ is decreasing in $n$ whereas $\YY^{x_n,y}$ is increasing in $n$ thanks to \eqref{dxX1} and \eqref{dxY1}. Then, by Proposition \ref{prop:c-ht} and the geometry of $\cS$ we have $\tau_*(x_n,y)\uparrow \tau_*(0,y)$, $\OP$-a.s. For simplicity we denote $\tau_n=\tau_*(x_n,y)$ and $\tau_\infty=\tau_*(0,y)$.

The idea is simply to take limits in the expressions of $\UU_x$ and $\UU_y$ (see \eqref{UUx} and \eqref{UUy}). 
For \eqref{UUx} we notice that $S^{\mu_0,\sigma}_{\tau_n}-x_n\uparrow S^{\mu_0,\sigma}_{\tau_\infty}$ as $n\to \infty$, and recall that $\OP(S^{\mu_0,\sigma}_{\tau_\infty}= 0)=0$ by \eqref{pt0t*}, since $y>y^*_0$. Then $\OP$-a.s.~we have
\begin{align}
\label{l1}&\lim_{n\to\infty}\mathds{1}_{\{S^{\mu_0,\sigma}_{\tau_n}-x_n\ge 0\}}=\mathds{1}_{\{S^{\mu_0,\sigma}_{\tau_\infty}\ge 0\}}=1,\\[+4pt]
\label{l2}&\lim_{n\to\infty}\mathds{1}_{\{S^{\mu_0,\sigma}_{\tau_n}-x_n\le 0\}}=\mathds{1}_{\{S^{\mu_0,\sigma}_{\tau_\infty}\le 0\}}=0.
\end{align}

Once again we use \eqref{simple} to invoke dominated convergence theorem, upon noticing that $S^{\beta,\sigma}_{\tau_n}\le S^{\beta,\sigma}_{\tau_\infty}$ for any $\beta$. From \eqref{l1}--\eqref{l2} we then obtain (restoring the notation $\tau_\infty=\tau_*$ under $\OP_{0,y}$)
\begin{align*}
\UU_x(0+,y)=&\,\lim_{n\to\infty}\UU_x(x_n,y)\nonumber\\[+4pt]
=&-\tfrac{2\mu_0}{\sigma^2}\OE_{0,y}\left[\mathds{1}_{\{\tau_*<+\infty\}}e^{\frac{2\mu_0}{\sigma^2}A_{\tau_*}-\rho\tau_*}\left(1+e^{\frac{\theta}{\sigma}(\XX_{\tau_*}+\YY_{\tau_*})}\right)\right]\nonumber\\[+4pt]
&-\tfrac{\theta}{\sigma}\OE_{0,y}\left[\mathds{1}_{\{\tau_*<+\infty\}}e^{\frac{2\mu_0}{\sigma^2}A_{\tau_*}-\rho\tau_*}e^{\frac{\theta}{\sigma}(\XX_{\tau_*}+\YY_{\tau_*})}\right]\\[+4pt]
=&-\OE_{0,y}\left[\mathds{1}_{\{\tau_*<+\infty\}}e^{\frac{2\mu_0}{\sigma^2}A_{\tau_*}-\rho\tau_*}\left(\tfrac{2\mu_0}{\sigma^2}+\tfrac{2\mu_0}{\sigma^2}e^{\frac{\theta}{\sigma}(\XX_{\tau_*}+\YY_{\tau_*})}+\tfrac{\hat \mu}{\sigma^2}e^{\frac{\theta}{\sigma}(\XX_{\tau_*}+\YY_{\tau_*})}\right)\right].\nonumber
\end{align*}

Similarly, for $\UU_y$ we get
\begin{align*}
\UU_y(0+,y)=&\,\lim_{n\to\infty}\UU_y(x_n,y)\nonumber\\[+4pt]
=&\tfrac{\hat \mu}{\sigma^2}\OE_{0,y}\left[\mathds{1}_{\{\tau_*<+\infty\}}e^{\frac{2\mu_0}{\sigma^2}A_{\tau_*}-\rho\tau_*}e^{\frac{\theta}{\sigma}(\XX_{\tau_*}+\YY_{\tau_*})}\right].
\end{align*}

Combining the two expressions above we find
\begin{align*}
&\tfrac{\sigma^2}{2}(\UU_x+\UU_y)(0+,y)\\
&=-\mu_0\OE_{0,y}\left[\mathds{1}_{\{\tau_*<+\infty\}}e^{\frac{2\mu_0}{\sigma^2}A_{\tau_*}-\rho\tau_*}\left(1+e^{\frac{\theta}{\sigma}(\XX_{\tau_*}+\YY_{\tau_*})}\right)\right]=-\mu_0\UU(0,y),
\end{align*}
where the last equality uses \eqref{opt-1}.
\end{proof}

With the aim of eventually going back to our original problem \eqref{U} in the $(x,\pi)$ coordinates, we need now to consider the inverse of $b(\cdot)$. In particular, recalling the non-decreasing map $x\mapsto\chi(x)$ from \eqref{xi}, and noticing that 
\begin{align*}
x<b(y)\iff y>\chi(x),
\end{align*}
we conclude that $\chi$ is the right continuous inverse of $b$, i.e.
\begin{align*}
\chi(x)=\inf\{y\in\R\,:\,b(y)>x\}.
\end{align*}
From \eqref{xi} we also obtain that the map $x\mapsto \psi(x)$ is non-decreasing and right-continuous with
\begin{align*}
\psi(x)=\exp\left[\tfrac{\theta}{\sigma}\big(\chi(x)+x\big)\right].
\end{align*}

We can therefore take the non-decreasing, left-continuous inverse of $\psi$, that is
\begin{align}\label{c}
c(\varphi)=\inf\{x>0\,:\,\psi(x)\ge \varphi\}
\end{align}
and notice that
\[
\varphi>\psi(x)\iff x<c(\varphi).
\]

Next, we recall that $\varphi=\pi/(1-\pi)$ and, since $\pi\mapsto \pi/(1-\pi)$ is increasing, we can define the optimal boundary in the $(x,\pi)$-coordinates by setting 
\begin{align}\label{d}
d(\pi):=c\left(\tfrac{\pi}{1-\pi}\right)\left(=c(\varphi)\right).
\end{align}
Clearly $\pi\mapsto d(\pi)$ is left-continuous and non-decreasing and, finally, we can define its right-continuous, non-decreasing inverse
\begin{align*}
\lambda(x):=\inf\{\pi\in(0,1)\,:\,d(\pi)>x\}.
\end{align*}

Summarising the above, the sets $\cC$ and $\cS$ can be equivalently described in terms of $d(\cdot)$, $\lambda(\cdot)$, $c(\cdot)$, $\psi(\cdot)$, $b(\cdot)$ or $\chi(\cdot)$, depending on the chosen coordinates, i.e.
\begin{align}\label{C-last}
\cC=&\{(x,y)\,:\,y>\chi(x)\}=\{(x,y)\,:\,x<b(y)\}\nonumber\\
=&\{(x,\varphi)\,:\,\varphi>\psi(x)\}=\{(x,\varphi)\,:\,x<c(\varphi)\}\\
=&\{(x,\pi)\,:\,\pi>\lambda(x)\}=\{(x,\pi)\,:\,x<d(\pi)\},\nonumber
\end{align} 
and 
\begin{align}\label{S-last}
\cS=&\{(x,y)\,:\,y\le \chi(x)\}=\{(x,y)\,:\,x\ge b(y)\}\nonumber\\
=&\{(x,\varphi)\,:\,\varphi\le \psi(x)\}=\{(x,\varphi)\,:\,x\ge c(\varphi)\}\\
=&\{(x,\pi)\,:\,\pi\le \lambda(x)\}=\{(x,\pi)\,:\,x\ge d(\pi)\}.\nonumber
\end{align} 

Before closing this section we determine the limiting behaviour of the boundary $d(\pi)$ as $\pi\to\{0,1\}$. Let us recall the measure $\P^\theta$ introduced in the proof of Proposition \ref{prop:OU1n} and the associated Brownian motion $W^\theta$. Moreover let us also consider 
\begin{align}\label{Umu1}
U^{\mu_1}(x)=\sup_{\tau\ge 0}\E^\theta_x\left[e^{\frac{2\mu_1}{\sigma^2}A_{\tau}-\rho\tau}\right]
\end{align}
which corresponds to problem \eqref{U} with $\pi=1$ (notice that indeed $\XX$ has drift $\mu_1$ under $\P^\theta$).
It was shown in \cite[Sec.~8.3]{DeAE17} that \eqref{Umu1} is the optimal stopping problem associated to the dividend problem with full information and  drift of $X^D$ equal to $\mu_1$. It then follows from \cite{DeAE17} that there is an optimal stopping boundary $a^*>0$ that fully characterises the solution of \eqref{Umu1} and the stopping set is $[a^*,+\infty)$ (an expression for $a_*$ can be found in \cite[Thm.~2.53, Ch.~2]{Sch} with the notation $m=\mu_1$ and $\delta=\rho$).

We now notice that using Girsanov theorem and \eqref{limn}, from \eqref{U-OU} we obtain
\begin{align}\label{Umu2}
U(x,\varphi/(1+\varphi))=&\frac{\OU(x,\varphi)}{1+\varphi}=\lim_{n\to\infty}\frac{\OU^n(x,\varphi)}{1+\varphi}\\
=&\lim_{n\to\infty}\frac{\varphi}{1+\varphi}\sup_{\tau\le \zeta_n}\left(\tfrac{1}{\varphi}\OE\left[e^{\frac{2\mu_0}{\sigma^2}A^x_\tau-\rho\tau}\right]+\E^\theta\left[e^{\frac{2\mu_1}{\sigma^2}A^x_\tau-\rho\tau}\right]\right)\nonumber\\
=&\frac{\varphi}{1+\varphi}\sup_{\tau\ge0}\left(\tfrac{1}{\varphi}\OE\left[e^{\frac{2\mu_0}{\sigma^2}A^x_\tau-\rho\tau}\right]+\E^\theta\left[e^{\frac{2\mu_1}{\sigma^2}A^x_\tau-\rho\tau}\right]\right).\nonumber
\end{align}
Then letting $\pi\to 1$ (or equivalently $\varphi\to\infty$) we obtain from the last expression above
\begin{align}\label{Umu3}
\lim_{\pi\to 1} U(x,\pi)=U^{\mu_1}(x),\quad\text{for all $x\in[0,+\infty)$}.
\end{align}

We also need to state two simple facts which can be obtained by \eqref{U-OU} and straightforward calculations. For all $(x,\pi)\in\cO$ we have
\begin{align*}
U_x(x,\pi)=\frac{1}{1+\varphi}\OU_x(x,\varphi),\quad U_\pi(x,\pi)=-\OU(x,\varphi)+(1+\varphi)\OU_\varphi(x,\varphi).
\end{align*}
Thanks to \eqref{sublin} and \eqref{Lip-OU}, the above and \eqref{U-OU} imply that there is a constant $c>0$ such that
\begin{align}\label{gradb}
|U(x,\pi)|+|U_x(x,\pi)|+(1-\pi)|U_\pi(x,\pi)|\le c, \quad\text{for $(x,\pi)\in\cO$}.
\end{align}

We can now state our next result. 
\begin{proposition}\label{prop:lims-d}
Under Assumption \ref{ass:mu} we have 
\begin{align}\label{lims-d}
\lim_{\pi\to 0}d(\pi)=0\quad\text{and}\quad\lim_{\pi\to 1} d(\pi)=a^*,
\end{align}
where $a^*$ is the optimal boundary for \eqref{Umu1}.
\end{proposition}
\begin{proof}
\emph{Step 1} (Limit as $\pi\to 1$).
Recall that $d(\cdot)$ is non-decreasing and left-continuous. Then 
\begin{align}\label{limd}
d(1)=\lim_{\pi\to1}d(\pi). 
\end{align}
Thanks to \eqref{gradb} we have
\begin{align*}
\left|U\big(d(\pi),\pi\big)-U^{\mu_1}\big(d(1)\big)\right|\le &\left|U\big(d(\pi),\pi\big)-U\big(d(1),\pi\big)\right|+\left|U\big(d(1),\pi\big)-U^{\mu_1}\big(d(1)\big)\right|\\[+3pt]
\le &\,c\,\big(d(1)-d(\pi)\big)+\left|U\big(d(1),\pi\big)-U^{\mu_1}\big(d(1)\big)\right|.
\end{align*}
Recall that $U\big(d(\pi),\pi\big)=1$ for all $\pi\in(0,1)$, hence taking limits as $\pi\uparrow 1$ in the expression above and using \eqref{Umu3} and \eqref{limd}, we obtain
\begin{align}
1=\lim_{\pi\to 1} U(d(\pi),\pi)=U^{\mu_1}(d(1)).
\end{align}
The latter implies $d(1)\ge a^*$ by definition of $a^*$.

Let us now assume that $d(1)>a^*$ and take an interval $[x_1,x_2]\subset (a^*,d(1))$. Pick an arbitrary positive function $\phi\in C^\infty_c(x_1,x_2)$ such that $\int_{\R_+}\phi(\zeta)=1$. Rewriting \eqref{BVP1} in the $(x,\pi)$-coordinates, we have $(\cL_{X,\pi}-\rho)U=0$ in $\cC$. Thanks to left-continuity of $d(\cdot)$ we can choose $\eps>0$, such that $\mathcal{R}_\eps:=[x_1,x_2]\times [1-\eps,1)\subset\cC$ and 
\[
\phi(x)(\cL_{X,\pi}U-\rho U)(x,\pi)=0\qquad \text{for $(x,\pi)\in\mathcal{R}_\eps$}.  
\]
Integration by parts gives
\begin{align}\label{d1}
0=&\int^{x_2}_{x_1}\phi(\zeta)(\cL_{X,\pi}-\rho)U(\zeta,\pi)d\zeta\\
=&\int^{x_2}_{x_1} U(\zeta,\pi)(\cG-\rho)\phi(\zeta)d\zeta\nonumber\\
&+\pi(1-\pi)\int^{x_2}_{x_1}\big(\tfrac{\theta^2}{2}\pi(1-\pi)U_{\pi\pi}(\zeta,\pi)+\hat\mu U_{x\pi}(\zeta,\pi)\big)\phi(\zeta) d\zeta\nonumber
\end{align} 
where $\cG=\tfrac{\sigma^2}{2}\tfrac{\partial^2}{\partial x^2}-(\mu_0+\hat\mu\pi)\tfrac{\partial}{\partial x}$. Set
\[
F_\phi(\pi):=\int^{x_2}_{x_1}\big(\tfrac{\theta^2}{2}\pi(1-\pi)U_{\pi\pi}(\zeta,\pi)+\hat\mu U_{x\pi}(\zeta,\pi)\big)\phi(\zeta) d\zeta
\]
and let $\pi\to 1$ in \eqref{d1}. Then by \eqref{Umu3} and dominated convergence, we get
\begin{align*}
\lim_{\pi\to1}\pi(1-\pi)F_\phi(\pi)=-\int^{x_2}_{x_1} U^{\mu_1}(\zeta)(\cG-\rho)\phi(\zeta)d\zeta.
\end{align*}
Since $U^{\mu_1}(x)=1$ for $x\in(x_1,x_2)$, after undoing the integration by parts, we get
\begin{align}\label{d2}
\lim_{\pi\to1}\pi(1-\pi)F_\phi(\pi)=\rho.
\end{align}
Now, \eqref{d2} says that $F_\phi(\pi)$ behaves as $\rho/(1-\pi)$ for $\pi\to 1$. That implies 
\begin{align}\label{contr}
\int^1_{1-\eps}F_\phi(\pi)d\pi=+\infty.
\end{align}
We will show that \eqref{contr} is impossible.

For $\eps>0$ as above and $0<\delta<\eps$, Fubini's theorem and integration by parts give
\begin{align*}
\int^{1-\delta}_{1-\eps}&F_\phi(\pi)d\pi\\
=&\int_{x_1}^{x_2}\left[\int^{1-\delta}_{1-\eps}\big(\tfrac{\theta^2}{2}\pi(1-\pi)U_{\pi\pi}(\zeta,\pi)+\hat\mu U_{x\pi}(\zeta,\pi)\big)d\pi\right] \phi(\zeta)d\zeta\\
=&\tfrac{\theta^2}{2}\int_{x_1}^{x_2}\left[\Big(\pi(1-\pi)U_\pi(\zeta,\pi)-(1-2\pi)U(\zeta,\pi)\Big|^{\pi=1-\delta}_{\pi=1-\eps}-2\int_{1-\eps}^{1-\delta}U(\zeta,\pi)d\pi\right]\phi(\zeta)d\zeta\\
&+\hat\mu\int_{x_1}^{x_2} \Big(U_x(\zeta,\pi)\Big|^{\pi=1-\delta}_{\pi=1-\eps}\,\,\,\phi(\zeta)d\zeta\le  c',
\end{align*}
where the last inequality uses \eqref{gradb} and $c'>0$ is independent of $\delta$. Letting $\delta\to 0$ we reach a contradiction with \eqref{contr}.
\vspace{+4pt}

\emph{Step 2} (Limit as $\pi\to 0$). The proof follows the same steps as above. Let $d(0+):=\lim_{\pi\to 0}d(\pi)$ and assume that $d(0+)>0$. Then take $[x_1,x_2]\subset (0,d(0+))$ and an arbitrary positive function $\phi\in C^\infty_c(x_1,x_2)$ such that $\int_{\R_+}\phi(\zeta)=1$. Repeating the same steps as above we write \eqref{d1} and notice that (iii) in Proposition \ref{prop:OU1} implies that $\lim_{\pi\to0}U(x,\pi)=1$ for all $x\ge 0$. Hence taking $\pi\to0$ in \eqref{d1} gives
\begin{align}\label{d3}
\lim_{\pi\to0}\pi(1-\pi)F_\phi(\pi)=\rho,
\end{align}
which also implies $\int^{\eps}_0F_\phi(\pi)d\pi=+\infty$. The latter leads to a contradiction, exactly as in step 1 above.
\end{proof}

Using \eqref{C-last} and \eqref{S-last} we can conclude that also the boundaries $c$ and $b$ are bounded above by $a_*$ and have the same limits.
\begin{corollary}
We have $0\le c(\varphi)\le a_*$ for $\varphi\in\R_+$ and $0\le b(y)\le a_*$ for $y\in\R$. Moreover
\[
\lim_{\varphi\to 0}c(\varphi)=\lim_{y\to-\infty}b(y)=0\quad\text{and}\quad\lim_{\varphi\to\infty}c(\varphi)=\lim_{y\to\infty}b(y)=a_*.
\]
\end{corollary}


\section{Solution of the dividend problem}\label{sec:solution}

At this point we can construct a candidate for the value function $V$ in \eqref{P2} by setting
\begin{align}\label{v}
v(x,\pi):=\int_0^x U(\zeta,\pi)d\zeta,\qquad(x,\pi)\in\overline \cO.
\end{align} 
Thanks to Corollary \ref{cor:vC1} and dominated convergence we immediately obtain
\begin{corollary}\label{cor:U-C2}
Under Assumption \ref{ass:mu} the function $v$ belongs to $C(\overline{\cO})\cap C^1(\cO)$. Moreover $v_{xx}$ and $v_{x\pi}$ are continuous in $\cO$.
\end{corollary}
In order to apply Theorem \ref{thm:verif}, it remains to show that $v_{\pi\pi}\in L^\infty_{loc}(\cO)$ and $v_{\pi\pi}\in C(\overline{\cC}\cap\cO)$. This is a non-trivial task and relies on a semi-explicit characterisation of the weak derivative $v_{\pi\pi}$.
\begin{proposition}\label{prop:vpipi}
Let Assumption \ref{ass:mu} hold. The function $v$ in \eqref{v} admits weak derivative $v_{\pi\pi}\in L^\infty_{loc}(\cO)$. Moreover, we can select an element of the equivalence class of $v_{\pi\pi}\in L^\infty_{loc}(\cO)$ (denoted again by $v_{\pi\pi}$) given by  
\begin{align}\label{vpipi}
v_{\pi\pi}(x,\pi)=2&\Big[\rho\!\!\int_0^{x\wedge d_+(\pi)}U(\zeta,\pi)d\zeta-\frac{\sigma^2}{2}U_x(x\wedge d_+(\pi),\pi)\\[+4pt]
&\quad-\hat\mu\pi(1-\pi)U_\pi(x\wedge d_+(\pi),\pi)\notag\\
&\quad-(\mu_0+\hat\mu\pi)U(x\wedge d_+(\pi),\pi)\Big]\big(\theta\pi(1-\pi)\big)^{-2},\notag
\end{align}
with $d_+(\pi):=\lim_{\eps\to 0}d(\pi+\eps)$.
\end{proposition} 
\begin{proof}
Since $v_\pi(x,\,\cdot\,)$ is a continuous function for all $x>0$, as usual we say that its weak derivative with respect to $\pi$ is a function $f\in L^1_{loc}(\cO)$ such that for any $\phi\in C^\infty_c(0,1)$ it holds
\begin{align}\label{weak}
\int_0^1v_\pi(x,z)\phi'(z)dz=-\int_0^1f(x,z)\phi(z)dz.
\end{align}
Our aim is to compute $f$, show that it equals the right-hand side of \eqref{vpipi} and therefore conclude that $f\in L^\infty_{loc}(\cO)$, due to $U\in C^1(\cO)$.

Recalling that $U_\pi=0$ in $\cS$ and that $x<d(\pi)\iff \pi> \lambda(x)$ (cf.~\eqref{C-last}), using Fubini's theorem we can write
\begin{align}\label{a00}
&\int_0^1v_\pi(x,z)\phi'(z)dz=\int_0^1 \left(\int_0^{x\wedge d(z)}U_\pi(\zeta,z)d\zeta \right)\phi'(z)dz\\
&=\int_0^x \left(\int_{\lambda(\zeta)}^1 U_\pi(\zeta,z)\phi'(z)dz \right)d\zeta\nonumber\\
&=\int_0^x\left(U_\pi(\zeta,1)\phi(1)-U_\pi(\zeta,\lambda(\zeta))\phi(\lambda(\zeta))
-\int_{\lambda(\zeta)}^1 U_{\pi\pi}(\zeta,z)\phi(z)dz \right)d\zeta\nonumber\\
&=-\int_0^x\left(\int_{\lambda(\zeta)}^1 U_{\pi\pi}(\zeta,z)\phi(z)dz \right)d\zeta,\nonumber
\end{align}
where the final equation holds because $U_\pi(\zeta,\lambda(\zeta))=0$ for all $\zeta\in(0,x)$ and $\phi(1)=0$.

Now we rewrite the last expression above by using that  
\begin{align*}
\tfrac{\theta^2}{2}\pi^2(1-\pi)^2 U_{\pi\pi}=-\tfrac{\sigma^2}{2}U_{xx}-\hat\mu\pi(1-\pi)U_{x\pi} -(\mu_0+\hat \mu\pi)U_x+\rho U\quad\text{in $\cC\setminus(\{0\}\times\R_+)$},
\end{align*}
thanks to \eqref{BVP1} written in the $(x,\pi)$-coordinates. Hence, by using Fubini's theorem again, we get
\begin{align}\label{a01}
-&\int_0^x\left(\int_{\lambda(\zeta)}^1 U_{\pi\pi}(\zeta,z)\phi(z)dz \right)d\zeta\\
=&2\int_0^1\!\!\Big(\int_0^{x\wedge d(z)}\Big[\tfrac{\sigma^2}{2}U_{xx}(\zeta,z)+\hat\mu z(1-z)U_{x\pi}(\zeta,z)\nonumber\\
&\hspace{+80pt} +(\mu_0+\hat \mu z)U_x(\zeta,z)-\rho U(\zeta, z)\Big]d\zeta\Big)\big[\theta z(1-z)\big]^{-2}\phi(z)dz.\nonumber
\end{align}

Let us now consider the integral with respect to $\zeta$ and notice that we need only be concerned with $z\in[0,1]$ such that $d(z)>0$, as otherwise the integral is zero. Using \eqref{HJBu3b} for $U$ we obtain
\begin{align}\label{a02}
\int_0^{x\wedge d(z)}&\Big[\tfrac{\sigma^2}{2}U_{xx}(\zeta,z)+\hat\mu z(1-z)U_{x\pi}(\zeta,z)+(\mu_0+\hat \mu z)U_x(\zeta,z)\Big]d\zeta\nonumber\\
=&\tfrac{\sigma^2}{2}\big[U_x(x\wedge d(z),z)-U_x(0+,z)\big]+\hat\mu z(1-z)\big[U_\pi(x\wedge d(z),z)-U_\pi(0+,z)\big]\nonumber\\[+4pt]
&+(\mu_0+\hat \mu z)\big[U(x\wedge d(z),z)-U(0+,z)\big]\\[+4pt]
=&\tfrac{\sigma^2}{2}U_x(x\wedge d(z),z)+\hat\mu z(1-z)U_\pi(x\wedge d(z),z)+(\mu_0+\hat \mu z)U(x\wedge d(z),z).\nonumber
\end{align}

Combining \eqref{a00}--\eqref{a02} we get
\begin{align*}
&\int_0^1v_\pi(x,z)\phi'(z)dz\\
&=2\int_0^1\Big(\tfrac{\sigma^2}{2}U_x(x\wedge d(z),z)+\hat\mu z(1-z)U_\pi(x\wedge d(z),z)\\
&\hspace{+50pt}+(\mu_0+\hat \mu z)U(x\wedge d(z),z)-\rho\int_0^{x\wedge d(z)}U(\zeta,z)d\zeta\Big)\big[\theta z(1-z)\big]^{-2}\phi(z)dz,
\end{align*}
from which we deduce 
\begin{align*}
f(x,\pi)=2\Big(&\rho\int_0^{x\wedge d(\pi)}U(\zeta,\pi)d\zeta-\tfrac{\sigma^2}{2}U_x(x\wedge d(\pi),\pi)\\
&-\hat\mu \pi(1-\pi)U_\pi(x\wedge d(\pi),\pi)-(\mu_0+\hat \mu \pi)U(x\wedge d(\pi),\pi)\Big)\big[\theta \pi(1-\pi)\big]^{-2}.
\end{align*}

Finally, notice that $\pi\mapsto d(\pi)$ has at most countably many jumps for $\pi\in [0,1]$, hence $f(x,\pi)=\lim_{\eps\to 0}f(x,\pi+\eps)$ for a.e.~$\pi\in[0,1]$. Moreover, let $(\pi^J_k)_{k\ge 1}$ be the collection of jump points of $d$ and denote 
\[
\mathcal{N}:=\bigcup_{k\ge 1}\Big([d(\pi^J_k),\infty)\times\{\pi^J_k\}\Big). 
\]
Then 
\[
f(x,\pi)=\lim_{\eps\to 0}f(x,\pi+\eps)\quad\text{for $(x,\pi)\in\cO\setminus\mathcal{N}$}.
\]
Since $\mathcal{N}$ has zero Lebesgue measure in $\cO$, we conclude that \eqref{vpipi} holds.
\end{proof}

In the remainder of the paper we will always consider the representative of $v_{\pi\pi}$ given by the expression in \eqref{vpipi}. From \eqref{vpipi} and $U\in C^1(\cO)$ we derive the next result.
\begin{corollary}\label{cor:vpipi2}
Under Assumption \ref{ass:mu} the function $v_{\pi\pi}$ in \eqref{vpipi} is continuous in $\overline{\cC}\cap\cO$.
\end{corollary}
\begin{proof}
It is sufficient to notice that for any $(x,\pi)\in\overline{\cC}\cap\cO$ we have $x\le d_+(\pi)$. Hence
\begin{align*}
v_{\pi\pi}(x,\pi)=2&\Big[\rho\!\!\int_0^{x}U(\zeta,\pi)d\zeta-\frac{\sigma^2}{2}U_x(x,\pi)\nonumber\\[+4pt]
&-\hat\mu\pi(1-\pi)U_\pi(x,\pi)-(\mu_0+\hat\mu\pi)U(x,\pi)\Big]\big(\theta\pi(1-\pi)\big)^{-2},
\end{align*}
for all $(x,\pi)\in\overline{\cC}\cap\cO$. Continuity of $v_{\pi\pi}$ now follows from $U\in C^1(\cO)$. 
\end{proof}

Now that we have a candidate solution for the variational problem in Theorem \ref{thm:verif}, we would like also to construct a candidate optimal control. Recalling $\cI_v$ as in \eqref{Iv} and noticing that $v_x=U$ we immediately see that 
$\cI_v=\cC$. Then, given $(x,\pi)\in\cO$ we define $\P_{x,\pi}$-a.s.~the process
\begin{align}\label{D}
\DD_t:=\sup_{0\le s\le t}\left[X_s-d(\pi_s)\right]^+
\end{align}
where we recall that $X$ is the uncontrolled dynamic
\[
X_t=x+\int_0^t(\mu_0+\hat \mu\pi_s)ds+\sigma W_t, \quad\text{$\P_{x,\pi}$-a.s.}
\]
We also recall the notation $\gamma^{\DD}:=\inf\{t\ge 0\,:\,X^{\DD}_t\le 0\}$.

Some of the arguments in the proof of the next lemma are borrowed from \cite[Sec.~5]{DeAFeFe17}. 
\begin{lemma}\label{lem:D}
Let Assumption \ref{ass:mu} hold. The process $\DD$ in \eqref{D} belongs to $\cA$ (i.e.~it is admissible). The treble $(X^{\DD}_t, \DD_t, \pi_t)_{t\ge0}$ solves the Skorokhod reflection problem in $\cC$, that is, $\P_{x,\pi}$-a.s.~for all $0\le t\le \gamma^{\DD}$ we have 
\begin{align}
\label{SK1b}&(X^{\DD}_t, \pi_t)\in\overline \cC,\\
\label{SK2b}&d\DD_t=\mathds{1}_{\{(X^{\DD}_{t-}, \pi_t)\notin \cC\}}d\DD_t,\\
\label{SK3b}&\int_0^{\Delta \DD_t}\mathds{1}_{\{(X^{\DD}_{t-}-z,\pi_t)\in \,\cC\}}dz=0.
\end{align}
\end{lemma}
\begin{proof}
It is immediate to see that $\DD$ is increasing and adapted to $(\cF_t)_{t\ge0}$. Since it is increasing then it also admits left limits at all points. In order to prove right-continuity of paths we observe that $d(\cdot)$ is non-decreasing and left continuous, hence lower semi-continuous. It then follows that the mapping $t\mapsto X_t-d(\pi_t)$ is $\P_{x,\pi}$-a.s.~upper semi-continuous. Now, obviously $\lim_{\eps\to 0}\DD_{t+\eps}\ge \DD_t$, and the reverse inequality follows by
\begin{align*}
\lim_{\eps\to 0}\DD_{t+\eps}=&\lim_{\eps\to 0}\left(\DD_t\vee\sup_{t<s\le t+\eps}\left[X_s-d(\pi_s)\right]^+\right)\\
=&\DD_t\vee\limsup_{\eps\to 0}\left[X_{t+\eps}-d(\pi_{t+\eps})\right]^+\le \DD_t\vee\left[X_t-d(\pi_t)\right]^+=\DD_t. 
\end{align*}
Hence $\DD\in\cA$.

Let us turn to the study of the Skorokhod reflection problem. Notice that, since $\pi$ is unaffected by $\DD$ we have 
\[
d(\pi_t)-X^\DD_t=d(\pi_t)-X_t+\DD_t\ge 0, \quad\text{for all $t\ge 0$, $\P_{x,\pi}$-a.s.,}
\]
where the final inequality follows from \eqref{D}. Recalling that $x< d(\pi)\iff (x,\pi)\in\cC$ we deduce that \eqref{SK1b} holds. It remains to prove \eqref{SK2b}.

Fix $\omega\in\Omega$ (outside of a null set), and fix $t_1>0$. If $X^\DD_{t_1}(\omega)<d\big(\pi_{t_1}(\omega)\big)$, then $X^{\DD}=X-\DD$ implies
\begin{align}\label{lbD}
\DD_{t_1}(\omega)>X_{t_1}(\omega)-d\big(\pi_{t_1}(\omega)\big).
\end{align} 
Due to upper semi-continuity of the map $t\mapsto X_t-d(\pi_t)$ and \eqref{lbD} there is $\eps_{\omega}:=\eps(\omega,t_1)$ such that
\[
\sup_{t_1<s\le t_1+\eps_{\omega}}\left[X_{s}(\omega)-d\big(\pi_{s}(\omega)\big)\right]^+\le \DD_{t_1}(\omega).
\]
Hence, for all $s\in[t_1,t_1+\eps_\omega]$ we have
\[
\DD_s(\omega)=\DD_{t_1}(\omega)\vee\sup_{t_1<s\le t_1+\eps_{\omega}}\left[X_{s}(\omega)-d\big(\pi_{s}(\omega)\big)\right]^+=\DD_{t_1}(\omega),
\]
which proves \eqref{SK2b} for all $0<t\le \gamma^\DD$. By right-continuity the result extends to  $0\le t\le \gamma^\DD$.
Finally, it follows from \eqref{D} that jumps of $\DD$ may only occur along vertical jumps of the boundary $d$, hence \eqref{SK3b} holds. 
\end{proof}

We can finally conclude the section by providing the solution of the dividend problem with partial information.
\begin{theorem}\label{thm:sol}
Recall $V$ from \eqref{P2} and $\DD$ from \eqref{D} and let Assumption \ref{ass:mu} hold. Then we have
\begin{align}\label{VV}
V(x,\pi)=\int_0^x U(\zeta,\pi)d\zeta,\qquad (x,\pi)\in\overline\cO
\end{align}
and $D^*=\DD$ is an optimal control.
\end{theorem}
\begin{proof}
We need to check that $v$ in \eqref{v} fulfills the assumptions of Theorem \ref{thm:verif}. It is immediate that $0\le v(x,\pi)\le c\,x$ thanks to \eqref{gradb}, hence $v(0,\pi)=0$. Moreover, Corollary \ref{cor:U-C2}, Proposition \ref{prop:vpipi} and Corollary \ref{cor:vpipi2} guarantee that $v$ is smooth enough. 

Next we verify that \eqref{HJB1} holds. Once again, notice that $\cI_v=\cC$ and let us pick $(x,\pi)\in \cC$. By direct calculation
\begin{align}\label{eq:sol1}
(\cL_{X,\pi}v-\rho v)=&\big[\tfrac{\sigma^2}{2}U_x+\hat \mu \pi(1-\pi)U_\pi+(\mu_0+\hat \mu \pi)U\big](x,\pi)\nonumber\\
&-\rho\int_0^x U(\zeta,\pi)d\zeta+\tfrac{\theta^2}{2}\pi^2(1-\pi)^2v_{\pi\pi}(x,\pi).
\end{align}
Then, substituting the expression \eqref{vpipi} for $v_{\pi\pi}$ in the above, and recalling that $(x,\pi)\in\cC$ was arbitrary, we obtain
\[
(\cL_{X,\pi}v-\rho v)(x,\pi)=0,\quad (x,\pi)\in\cC.
\]

Now, pick $(x,\pi)\in\cS$, recall that $U_x=U_\pi=0$ and $U=1$ in $\cS$ and repeat the calculations in \eqref{eq:sol1}. This gives
\begin{align*}
(\cL_{X,\pi}v-\rho v)=&(\mu_0+\hat \mu \pi)-\rho\int_0^x U(\zeta,\pi)d\zeta+\tfrac{\theta^2}{2}\pi^2(1-\pi)^2v_{\pi\pi}(x,\pi)\\
=&-\rho\int_{d(\pi)}^xU(\zeta,\pi)d\zeta=-\rho(x-d(\pi))\le 0,
\end{align*}
where we have used \eqref{vpipi}, upon noticing that $U_x(d(\pi),\pi)=U_\pi(d(\pi),\pi)=0$ and $U(d(\pi),\pi)=1$.

Finally, it was shown in Lemma \ref{lem:D} that \eqref{SK-0}, \eqref{SK-1} and \eqref{SK-2} hold with our choice of $D^*=\DD$. 
\end{proof}

\appendix
\section{Verification theorem}

\begin{proof}[\textbf{Proof of Theorem \ref{thm:verif}}]
Here we largely follow the proof in \cite[Thm.~4.1, Ch.~VIII]{FS}. Let $\phi\in C^\infty_c(B_1(0))$, where $B_1(0)$ is a ball in $\R^2$, centred in zero and with radius $1$. Moreover, assume $\phi\ge 0$ and $\int_{\R^2}\phi(z)dz=1$. For each $k\ge1$ we construct the standard mollifier $\phi_k(z):=k^{-2}\phi(kz)$ and notice that $\phi_k\in C^\infty_c(B_{1/k}(0))$. We then define a sequence $(v^k)_{k\ge 1}\subset C^{\infty}(\overline \cO)$, with 
\[
v^k(x,\pi):=(v*\phi_k)(x,\pi)=\int^\infty_{0}\int_0^1 v(\zeta,\eta)\phi_k(x-\zeta,\pi-\eta)d\eta d\zeta.
\] 
Thanks to the assumed regularity of $v$, for any compact $K\subset \cO$ we have 
\begin{align}
\label{conv-1}&\lim_{k\to\infty}\|v^k-v\|_{L^\infty(K)}=0,\\
\label{conv-2}&\lim_{k\to\infty}\left(\|v^k_x-v_x\|_{L^\infty(K)}+\|v^k_\pi-v_\pi\|_{L^\infty(K)}\right)=0,\\
\label{conv-3}&\lim_{k\to\infty}\left(\|v^k_{xx}-v_{xx}\|_{L^\infty(K)}+\|v^k_{x\pi}-v_{x\pi}\|_{L^\infty(K)}\right)=0.
\end{align} 
Moreover, we notice that since $v_{\pi\pi}\in L^\infty_{loc}(\cO)$, by the definition of weak derivative it is not hard to verify that $(v^k)_{\pi\pi}=(v_{\pi\pi}*\phi_k)$. 

Letting $K_n:=[n^{-1},n]\times[n^{-1},1-n^{-1}]$, thanks to the above (and continuity of the coefficients in $\cL_{X,\pi}$) it is easy to show that for any $n\ge 1$ we have
\begin{align}\label{conv-4}
\lim_{k\to\infty}\big\|(\cL_{X,\pi}v^k)-[(\cL_{X,\pi}v)*\phi_k]\big\|_{L^\infty(K_n)}=0.
\end{align}
Finally, since $(\cL_{X,\pi}-\rho) v\le 0$ a.e.~in $\cO$, then also $(\cL_{X,\pi}v-\rho v)*\phi_k \le 0$ everywhere~in $\cO$ and from \eqref{conv-4} we conclude
\begin{align}\label{conv-5}
\limsup_{k\to\infty}\sup_{(x,\pi)\in K_n}(\cL_{X,\pi}v^k-\rho v^k)(x,\pi)\le 0
\end{align}

Now fix $\eps>0$ and for any admissible control $D$, let $\gamma^D_\eps:=\inf\{t\ge0\,:\,X^D_t\le \eps\}$ and $\tau_n:=\inf\{t\ge 0\,:\, (X^D_t,\pi_t)\notin K_n\}$, $\P_{x,\pi}$-a.s. By an application of It\^o calculus, setting $\zeta_{\eps,n}:=\gamma^D_\eps\wedge\tau_n$ we derive
\begin{align}\label{verif0}
\E_{x,\pi}&\,\left[e^{-\rho(\zeta_{\eps,n}\wedge t)}v^k(X^D_{\zeta_{\eps,n}\wedge t},\pi_{\zeta_{\eps,n}\wedge t})\right]-v^k(x,\pi)\nonumber\\
=&\E_{x,\pi}\left[\int^{\zeta_{\eps,n}\wedge t}_{0}e^{-\rho s}(\cL_{X,\pi}-\rho)v^k(X^D_s,\pi_s)ds-\int^{\zeta_{\eps,n}\wedge t}_{0}e^{-\rho s}v^k_x(X^D_s,\pi_s)dD^{c}_s\right]\\
&+\E_{x,\pi}\Big[\sum_{s\le \zeta_{\eps,n}\wedge t}e^{-\rho s}\left(v^k(X^D_s,\pi_s)-v^k(X^D_{s-},\pi_s)\right)\Big],\nonumber
\end{align}
where $D^c$ denotes the continuous component of the process $D$.

Noticing that $(X^D_s,\pi_s)_{0\le s\le \zeta_{\eps,n}}\in K_n$ we can use \eqref{conv-1}--\eqref{conv-3} and \eqref{conv-5} to pass to the limit as $k\to\infty$ and obtain
\begin{align}\label{verif1}
&\E_{x,\pi}\left[e^{-\rho(\zeta_{\eps,n}\wedge t)}v(X^D_{\zeta_{\eps,n}\wedge t},\pi_{\zeta_{\eps,n}\wedge t})\right]\!-\!v(x,\pi)\\
&\le \!\E_{x,\pi}\Big[\!-\!\int^{\zeta_{\eps,n}\wedge t}_{0}\!\!e^{-\rho s}v_x(X^D_s,\pi_s)dD^{c}_s\!-\!\sum_{s\le \zeta_{\eps,n}\wedge t}\!e^{-\rho s}\!\int_0^{\Delta D_s}\!\!v_x(X^D_{s-}-z,\pi_s)dz\Big],\notag
\end{align}
where we have also used 
\[
v(X^D_s,\pi_s)-v(X^D_{s-},\pi_s)=-\int_0^{\Delta D_s}v_x(X^D_{s-}-z,\pi_s)dz.
\]
Using that $v_x\ge 1$, thanks to \eqref{HJB1}, and rearranging terms in \eqref{verif1} we get
\begin{align*}
v(x,\pi)\ge\E_{x,\pi}&\,\left[e^{-\rho(\zeta_{\eps,n}\wedge t)}v(X^D_{\zeta_{\eps,n}\wedge t},\pi_{\zeta_{\eps,n}\wedge t})\right]+\E_{x,\pi}\left[\int^{\zeta_{\eps,n}\wedge t}_{0}e^{-\rho s}dD_s\right].
\end{align*}
Letting $n\to\infty$, $\eps\to0$, $t\to\infty$ and recalling that $0\le v(x,\pi)\le c\,x$ we obtain 
\[
v(x,\pi)\ge \E_{x,\pi}\left[\int^{\gamma^D}_{0}e^{-\rho s}dD_s\right]. 
\]
Since $D$ is arbitrary, such inequality also implies $v\ge V$.

In order to show that $v\le V$, it is enough to observe that for $D=D^*$ all inequalities above become strict equalities. In particular, when taking limits in \eqref{verif0} we now use that $v\in C^2(\overline{\cI}_v\cap\cO)$ implies
\begin{align}\label{conv-6}
\lim_{k\to\infty}\sup_{(x,\pi)\in \overline{\cI}_v\cap K_n}\big|(\cL_{X,\pi}-\rho)(v^k-v)(x,\pi)\big| = 0.
\end{align} 
Also, we use that $(X^{D^*}_t,\pi_t)\in\overline{\cI}_v\cap K_n$ for all $0\le t\le \zeta_{\eps,n}$ and that $v_x(X^{D^*}_t,\pi_t)=1$ for all $t\in\text{supp}\{dD^*\}$.

\end{proof}
\vspace{+10pt}
\textbf{Acknowledgments}: This research was partially supported by the EPSRC grant EP/R021201/1, ``\emph{A probabilistic toolkit to study regularity of free boundaries in stochastic optimal control}''. Part of this work was done during a visit in the US at Rice University, Houston, Texas, in November 2017 and I thank P.~Ernst for the hospitality and Q.~Zhou and P.~Ernst for interesting discussions. I also thank E.~Ekstr\"om for the hospitality at the University of Uppsala in March 2018 and for suggesting the change of measure in Section \ref{sec:girsanov}. Finally, I would like to thank S.~Villeneuve for pointing me to useful references.


\end{document}